\documentclass[11pt]{amsart}
\usepackage{amsmath,amssymb,amsthm,graphicx, url, verbatim, float}
\usepackage{enumerate}
\usepackage{color, relsize}

\theoremstyle{definition}

\theoremstyle{plain}

\newcommand{\theoremname}{testing}
\newenvironment{named}[1]{\renewcommand{\theoremname}{#1}\begin{namedtheorem}}{\end{namedtheorem}}

\newcommand{\RR}{{\mathbb{R}}}
\newcommand{\ZZ}{{\mathbb{Z}}}
\newcommand{\NN}{{\mathbb{N}}}
\newcommand{\CC}{{\mathbb{C}}}

\newcommand{\sign}{{\rm sign}}

\newcommand{\vol}{{\rm vol}}

\newcommand{\abs}[1]{{\left\vert #1 \right\vert}}

\newcommand{\s}{\Sigma}

\newcommand{\CV}{C^{0}_P(\Sigma)}

\newcommand{\CE}{C^{1}_P(\Sigma)}

\newcommand{\dn}{ d_{qt}}

\newcommand{\dwp}{{ d_{WP}}}
\newcommand{\dpa}{ d_{\pi}}

 \newcommand{\M}{\mathrm{Mod}(\Sigma)}

\newcommand{\TS}{{\mathcal T}(\Sigma)}

\theoremstyle{plain}
\newtheorem{theorem}{Theorem}[section]
\newtheorem{corollary}[theorem]{Corollary}
\newtheorem{lemma}[theorem]{Lemma}
\newtheorem{proposition}[theorem]{Proposition}

\newtheorem{remark}[theorem]{Remark}

\newtheorem*{namedtheorem}{\theoremname}

\theoremstyle{definition}
\newtheorem{definition}[theorem]{Definition}

\title{Pants complex, TQFT and hyperbolic  geometry} 

\author{Renaud Detcherry}
\address{Institut de Mathématiques de Bourgogne, UMR 5584 CNRS, Université Bourgogne Franche-Comté, F-2100 Dijon, France
         }
\email{renaud.detcherry@u-bourgogne.fr}
\author{Efstratia Kalfagianni}
\address{Department of Mathematics, Michigan State University, East
Lansing, MI, 48824, USA}
\email{kalfagia@math.msu.edu}

\begin{document}

\date{\today}

\begin{abstract} We introduce a coarse perspective on relations of  the $SU(2)$-Witten-Reshetikhin-Turaev TQFT,  the   Weil-Petersson geometry of the
Teichm\"uller space, and  volumes of hyperbolic 3-manifolds. Using data from the  asymptotic expansions  of the curve operators in  the skein theoretic version of the $SU(2)$-TQFT, we define the quantum intersection number between pants decompositions of  a closed surface.
We show that the quantum intersection number  admits two sided
bounds in terms of the geometric intersection number and we use it to obtain
a metric on the pants graph of  surfaces. Using work of Brock we show that the pants graph equipped with this metric is quasi-isometric to the Teichm\"uller space with  the  Weil-Petersson metric and that the
translation length of our metric provides  two sided linear bounds on the volume of hyperbolic fibered manifolds. We briefly  discuss how these relations are interpeted
from the view point of $SU(2)$-character varieties of 3-manifolds.

We also obtain a characterization of pseudo-Anosov mapping classes in terms of   asymptotics of the quantum intersection number
under iteration in the mapping class group and relate these asymptotics with  stretch factors. We also discuss how these results fit with
a conjecture of Andersen, Masbaum and Ueno about  quantum representations of mapping class groups.

 \end{abstract}


\maketitle

\section{Introduction}\label{sec:intro} The purpose of the paper is to present a coarse perspective on relations of  asymptotic aspects of the unitary $SU(2)$-Witten-Reshetikhin-Turaev TQFT  to hyperbolic geometry. These connections are motivated and facilitated by the fact that  pant decompositions of surfaces play prominent roles in the
 skein theoretic construction of the $SU(2)$-TQFT, given by
Blanchet, Habegger, Masbaum and Vogel  \cite{BHMV, BHMV2},  and in  the construction of a combinatorial model for the   Weil-Petersson geometry on  Teichm\"uller  spaces given by Brock 
\cite{Brockpants}.

A \emph{pants decomposition} $P$ of a closed oriented  surface $\Sigma$ of genus $g>1$ is a collection of $3g-3$ distinct isotopy classes of simple closed curves on $\Sigma$.
The {\emph{pants graph}}  $\CE$, first considered by Hatcher and Thurston\cite{HT}, is an abstract graph with set of vertices $\CV$ isotopy classes of pants decompositions. Two vertices $P$
and $P'$ are connected by an 
edge in  $\CE$ if $P'$ is obtained from $P$ 
by an \emph{elementary move}. See
Definition \ref{def:emoves}.
 The pants graph  plays important roles in the study of  mapping class groups,  Teichm\"uller  theory, and hyperbolic geometry of surfaces.

For a positive integer $r$ and  a primitive $2r$-root of unity,  the $SU(2)$-Witten-Reshetikhin-Turaev TQFT  \cite{Turaevbook, BHMV2} assigns to $\Sigma$ a finite dimensional Hermitian vector space $V_r(\Sigma)$.
Given a pants decomposition $P,$ for any $r\geq 3$,  \cite{BHMV2} constructs a basis ${\mathcal B}_P:={\mathcal B}^r_P$ of $V_r(\Sigma)$ that is orthonormal with respect to the Hermitian pairing.
For any $Q\in \CV$,
they also construct a family of Hermitian  operators  $T^{Q}_r:  V_r(\Sigma)\longrightarrow V_r(\Sigma),$ known as \emph{curve operators},  so that  ${\mathcal B}_Q$ consists
of eigenvectors of the operator $T^{Q}_r$. 

In this paper we  work with the $TQFT$ theory associated to the roots of unity  $\zeta_r=-e^{\frac{\pi i}{2r}}$.
In this case, the large-$r$  asymptotic behavior of the curve operators was studied by Detcherry
in
\cite{RD}. His results imply that given $Q,P$
 there is an integer $n:=n(Q,P)\geq 0$, so that, for all $r$ larger enough,  the matrix representing  $T^{Q}_r$ in the basis ${\mathcal B}_P$
 has exactly $n$ diagonals  (excluding the main diagonal)  that contain non-zero entries. We call $n(Q,P)$  the \emph{quantum intersection number} of $Q$ with respect to $P$.

Since pants decompositions of surfaces play important roles in both  hyperbolic geometry and in TQFT, it is natural to ask what geometric properties of the pants graph are captured by quantities or invariants arising from TQFT.  
In this paper, we show that the quantum intersection number captures  the coarse geometry relations of the pants graph  to the Teichm\"uller space established in  \cite{Brockpants}.

The spectral radii of the operators $T^{Q}_r$ are uniformly bounded in terms of the genus of $\Sigma$,
and their norms  only depend on the orbit of $Q$ under the action of the mapping class group and not on the Hermitian bases of the spaces $V_r(\Sigma)$.
Intuitively speaking, this suggests that relations of large-$r$ asymptotics  of $T^{Q}_r$  with the geometry of $\CE$, should be reflected in the pairing defined
by the quantum intersection number. As we will discuss in Section \ref{sec:norms}, relations of the quantum intersection number to geometry, also translate to information about the geometric content  of the large-$r$ asymptotics of the curve operator norms.

We consider the quantum intersection number 
as a function $n: \CV\times \CV \longrightarrow \NN$
and 
we use a general process we call
\emph{metrification} to promote it to a metric $d_{qt}$ on $\CV$ such 
 that  the mapping class $\M$ acts  by isometries on  $(\CV, \dn)$. Roughly speaking, the distance  $\dn(P, Q)$ is the smallest total
quantum intersection number over all  finite sequences in $\CV$ from $P$ to $Q$.

 We prove the following theorem showing that the 0-skeleton of  pants graph  equipped with the metric $\dn$ derived from TQFT,  coarsely records the geometry of
Teichm\"uller space  $\TS$ equipped with its 
   Weil-Petersson  metric  $d_{WP}$.
   
\begin{theorem}\label{thm:QI} The metric space $\CV$ with the metric $\dn$ is quasi-isometric to the Teichm\"uller space with its Weil-Petersson  metric $d_{WP}$.

\end{theorem}   
   
   The \emph{Bers constant} $L=L(g)$ is a  number so that  any hyperbolic surface $X$ topologically equivalent to $\s$, admits  at least one pants decomposition $P_X$
   with the length of all curves at most $L$. By construction, Theorem  \ref{thm:QI} gives that  for any $P, Q\in \CV$, the distance $\dn(P_X, P_Y)$ is within bounded ratio from the distance $d_{WP}(X,Y)$, and the bounds depend only on $\s$.
 For more details, and for the definitions of terms and objects in the statement of Theorem \ref{thm:QI}, the reader is referred to Section \ref{sec:TSconnections}. 
        
 Theorem \ref{thm:QI} leads to new relations of asymptotic elements of the Witten-Reshetikhin-Turaev  $SU(2)$-TQFT to volumes of hyperbolic 3-manifolds.
Recall that  the translation length
 of an isometry $\phi$ on a metric space $(X, \ d)$ is defined by
 $$L^{d}(\phi)={\rm{inf}} \{ d(\phi(x), x)\ \ | \ \ x\in X\}.$$
 Since mapping classes $\phi \in \M$ act as isometries on $(\CV, \ \dn)$ and on $(\TS, \ \dwp )$ we can consider the translation lengths
 $ L^{\dn}(\phi)$ and $L^{\dwp}(\phi)$, respectively.
 Theorem \ref{thm:QI} implies that for any  $\phi \in \M$ the translation lengths
 $ L^{\dn}(\phi)$ and $L^{\dwp}(\phi)$ are within bounded ratios from each other, with bounds  depending only on the topology of $\s$.
 
 If $\phi$ is a pseudo-Anosov mapping class, the mapping torus $M_{\phi}:=F \times [0,1]/_{(x,1)\sim ({ {\phi}}(x),0)}$
 is a hyperbolic 3-manifold that fibers over the circle. In this case, Brock \cite{Brockvol} proved that the hyperbolic volume $\vol(M_{\phi})$ and
 the translation length  $L^{\dwp}(\phi)$ are  within bounded ratios from each other, with bounds depending only on  $\s$.
 Combining this result with Theorem \ref{thm:QI}   we obtain the following:
   
    \begin{theorem} \label{thm:volume}There exist  a positive  constant $N$, depending only on  $\s$,  so that for any  pseudo-Anosov mapping class $\phi \in \M$
 we have
 $${\frac{1}{N}}\ L^{\dn}(\phi)\leq \vol(M_{\phi}) \leq N\   L^{\dn}(\phi).$$
  \end{theorem}

Theorems \ref{thm:QI} and \ref{thm:volume} 
provide general new relations between  features of the $SU(2)$-TQFT  and hyperbolic geometry, approached from the view point  of coarse geometry of
 the Teichm\"uller space. In particular, Theorem  \ref{thm:volume} relates 2-dimensional TQFT aspects to 3-dimensional hyperbolic geometry which
 is compatible with the  TQFT philosophy that the quantum invariants of mapping tori are  determined by invariants of the corresponding monodromies.

Open  conjectures in quantum topology, predict  finer and exact relations between  features of the $\mathrm{SU}(2)$-TQFTs and   3-manifold  hyperbolic geometry.
For instance, the  volume conjecture of Kashaev asserts that  large-$r$ asymptotics of values of the colored Jones polynomial of hyperbolic knots determine the volume of the knot complement. More recently,
Chen and Yang \cite{Chen-Yang} conjectured that  similar asymptotics of the Witten-Reshetikhin-Turaev and the related Turaev-Viro invariants, determine the volumes of hyperbolic 3-manifolds. 
Similarly, a conjecture of Bonahon-Wong-Yang predicts that the volume of hyperbolic mapping tori is determined by asymptotics
of certain $SL_2(\CC)$-quantum invariants of pseudo-Anosov mapping classes \cite{BWY, BW}.
In the last couple of decades, 
work by several authors has verified and refined these conjectures  \cite{DKY, MR2797089, BDKY, Gromov, yangsurvey, KuM}  for special families of 3-manifolds 
but  general approaches seem to be currently missing. Theorems \ref{thm:QI} and \ref{thm:volume}, and corollaries stated in Section \ref{sec:TSconnections},
are compatible with and support the general picture of relations  predicted by these conjectures.


\subsection{Geometric and  quantum intersection numbers} 
Given an  integer  $r\geq 3$, let ${\mathcal C}_r:=\{ 1, 2,\ldots, r-1\}$ and let   $\zeta_r=-e^{\frac{\pi i}{2r}}$. Given a pants decomposition and a \emph{dual banded graph} $\Gamma$, with set of edges $E$, \cite{BHMV2}
considers  the set $U_r$ of  \emph{admissible colorings}, which are certain functions 
${\bf c}: E\longrightarrow {\mathcal C}_r$.
To an admissible coloring ${\bf c}$, they associate a vector $\phi_{{\bf c}}$ in  $V_r(\Sigma)$ so that
the set ${\mathcal B}_P=\{\phi_{{\bf c}}\}_{{\bf c}}$ is an orthonormal basis of $V_r(\Sigma)$.
For any ${\bf c}\in U_r$,  the function $\frac{\bf c}{r}$ is an element in the space ${\RR}^{E}$
of functions from $E$ to $(0,1)$.

For any admissible coloring ${\bf c}$, the pair
$(\frac{\bf c}{r},  \frac{ 1}{r})$ is an element 
of ${\RR}^{E}\times [0, 1]$.
Detcherry  \cite{RD} 
showed that,  for any  pants decompositions $Q$, $P$,   there is an open set $V_r\subset {\RR}^{E}\times [0, 1]$,
and analytic functions
$$G_{\bf k}^{Q}=G_{\bf k}^{Q}(\frac{\bf c}{r},  \frac{ 1}{r}) : V_{\gamma} \longrightarrow \CC,$$ parametrized by functions ${\bf  k}: E\longrightarrow \ZZ, $
so that 
\begin{equation}\label{eq:opcoefficients}
T^{Q}_r (\phi_{{\bf c}})=\underset{{\bf  k}}{\sum} G_{\bf k}^{Q}(\frac{c}{r},\frac{1}{r})\phi_{{{\bf c}+{ \bf k}}}.
\end{equation}
Furthermore, for given $P,Q$ and $\Gamma$,
the family $\{G_{\bf k}^{Q}\}_{{\bf k\neq 0}}$ contains finitely many non-zero functions. 
The number of non-zero functions, denoted by
 $n(Q, \ P)$, 
is the quantum intersection number of $Q$ with respect to $P$.

A key ingredient  in the proofs of Theorems \ref{thm:QI} and \ref{thm:volume} is the following result.

\begin{theorem}\label{thm:inequalities} Given pants decompositions $P$ and $Q$, let $I(Q,P)$ denote their total geometric intersection number. Then,
$$
\frac{I(Q,P)}{3g-3}\leq  n(Q, P)  \leq (I(Q,P)+1)^{3g-3}-1.
$$
\end{theorem}

We stress that the definition of $n(Q, P)$ is  not symmetric in $P, Q$. Nevertheless, using the process of metrification, described in Appendix
\ref{sec:metrification}, we define a metric $\dn$ on $\CV$ out of the quantum intersection number. 

The  \emph{path metric} $\dpa$ on the pants graph $\CE$   is defined by assigning the value 1 to each edge of $\CE$. Work of Mazur and Minsky
on the geometry of curve complexes, and Theorem \ref{thm:inequalities}  can be used to show that the distances $\dpa(P, Q)$ and $\dn(P, Q)$ are within bounded ratio 
from each other. See Theorem \ref{thm:PathTQFT}.
This, in turn, combined with Brock's result  stating that $(\CV, \  \dpa) $ is quasi-isometric with the Teichm\"uller space  $\TS$ equipped with its 
 Weil-Petersson  metric  $d_{WP}$,  gives Theorem \ref{thm:QI}.

The proof of Theorem \ref{thm:inequalities} relies on quantum topology techniques. A key part is to use
 the Masbaum-Vogel fusion theory underlying the $SU(2)$ skein theory to find \emph{state-sum} formulae for the functions $G_k^{Q}(\theta,0)$.
 The proof is given in Sections 4 and 5.

As a corollary of the proof of Theorem \ref{thm:volume} we get the following result.

 \begin{corollary} \label{cor:content}There  is a constant $N>0$, only depending  on the topology of $\s$, so that for any pseudo-Anosov  mapping class $\phi \in \M$ and any $P\in \CV$ we have 
 $$ n(P, \phi(P)) \geq N \  \vol(M_{\phi}).$$
 
 \end{corollary}
 
 The quantum intersection number
 $ n(\gamma, P)$ is defined for any multicurve $\gamma$ on $\s$ and $P\in \CV$, and Theorem \ref{thm:inequalities} holds. In this setting,
 it follows that quantum intersection numbers can be used to detect the Nielsen-Thurston types of mapping classes and to estimate stretch factors
 of pseudo-Anosov mapping classes.
 The reader is referred to  Section \ref{sec:TSconnections}.

\subsection{Relation to the AMU conjecture.} Recall that for any  $r\geq 3$ and  $\zeta_r=-e^{\frac{\pi i}{2r}}$ the $SU(2)$-TQFT also supplies a linear representation
$\rho_r$ of $\M$ on the space $V_r(\Sigma)$.
An intriguing open conjecture of Andersen, Masbaum and Ueno \cite{AMU} asserts that for any pseudo-Anosov $\phi \in \M$, the image $\rho_r(\phi)$ has infinite order for all large enough levels $r$.

For 
a multicurve $\gamma$ and $P\in \CV$, let $r_0(\gamma, P)>0$ denote the smallest integer such that for any lever $r\geq r_0(\gamma, P)$ 
  the matrix representing  $T^{\gamma}_r$ in the basis ${\mathcal B}_P$, realizes $n(\gamma, P)$.

\begin{corollary} \label{amu} Let $\phi$ be a  pseudo-Anosov mapping class such that there is  a multicurve $\gamma$ and $P\in \CV$
so that  $\underset{m\to\infty}{\limsup} \sqrt[m]{r_0(\phi^{m}(\gamma),P)} =l< \infty$. Suppose that $\phi$ fails the AMU conjecture. That is, there exists
an infinite sequence of positive integers $\{m_r\}_{r\to \infty}$ such that
$\rho_r((\phi)^{m_r})=1$.  Then,  $\underset{m_r\to\infty}{\limsup} \sqrt[m_r]{r} <l$.
\end{corollary}



To the best of our knowledge, at this writing, there isn't even a single example of a pseudo-Anosov mapping classes  acting on a closed surface for which
 the AMU conjecture is verified. The authors of this paper, used their work on the asymptotics of Turaev-Viro 3-manifold invariants to give constructions of vast families
 of pseudo-Anosov mapping classes, acting on any surface with boundary, that satisfy  the AMU conjecture \cite{BDKY, DK:AMU, DKI, KMe}. However, these techniques have not so far been successful for the case of closed surfaces. Corollary \ref{amu}, which follows from  the general version of Theorem \ref{thm:inequalities}, 
 shows that the problem can potentially  be approached
 through the study of quantum intersection numbers.
 
  \smallskip

 \subsection{Fourier coefficients of trace functions and volume} The works of  Sikora-Przytycki \cite{BFK, SikP} and Turaev \cite{Turskein} 
 provide a way to view the
  TQFT spaces $\{ V_r(\Sigma)\}_r$ as quantizations of the character variety 
$\mathcal{M}(\Sigma)=\mathrm{Hom}(\pi_1(\Sigma),\mathrm{SU}(2))/\mathrm{SU}(2)$ along the  Goldman Poisson structure.
In this setting,  the curve operators 
$\{T^{Q}_r\}_r$  form quantizations, in the sense of   \cite{RD}, of  the trace functions  $f_{Q}$ defined by $$f_{Q}(\rho)= \underset{\gamma \in Q}{\prod}\left(-\mathrm{Tr}(\rho(\gamma))\right),$$ for any $\rho \in \mathcal{M}(\Sigma)$.
More specifically, the composition and the commutator of
operators correspond to the product and  Goldman's Poisson bracket of the trace functions.
A pants decomposition $P,$ viewed as a family of $3g-3$ curves, induces a natural action of a torus on $\mathcal{M}(\Sigma),$  which is the hamiltonian flow of the moment map
 $$h_P: \rho \longmapsto \left(\frac{1}{\pi}\arccos(\mathrm{Tr}(\rho(\alpha)))\right)_{\alpha \in P},$$ defined in  \cite{G86}. Given $P,Q$, the trace function $f_{Q}$ has a Fourier decomposition with respect to the torus action defined by $P$ and
 by \cite{RD}, the coefficients $G_k^{Q}(\theta,0)$ correspond to the Fourier coefficients $f_{Q}$.
 In this setting, Theorem \ref{thm:inequalities}  shows that the number of non-zero Fourier coefficients of the trace functions $f_{Q}$
grows as the intersection number of $P, Q$ and it improves 
the bound of $2^{3g-3}$ that was previously given in  \cite{CM}.  
From this view point, Theorem \ref{thm:volume} reveals a previously unnoticed relation of  Goldman's picture to volumes of hyperbolic 3-manifolds,
while the proof relies on purely quantum topological techniques.
\smallskip

 \noindent{\bf{Acknowledgements.}}
We thank Dave Futer, Saul Schleimer and Sam Taylor for helpful discussions about references on Proposition \ref{pro:pathintersection}. 
During the course of this work, Detcherry was partially supported by the project “AlMaRe” (ANR-19-CE40-0001-01) and by the EIPHI Graduate School
(ANR-17-EURE-0002) and Kalfagianni was partially
supported by NSF grants DMS-1708249, DMS-2004155 and  DMS-2304033.

\section{TQFT curve operators and their coefficient functions} \label{sec:two}
In this section, after recalling the setting from \cite{BHMV2} and the results we need from \cite{RD},  we will give
the precise definition of the quantum intersection number,  state a general version of Theorem \ref{thm:inequalities} and outline the strategy of its proof.
\subsection{Preliminaries and definitions}
Let $\Sigma$ be a closed surface of genus $g\geq  1$.  A \emph{multicurve} on $\Sigma$ is a collection of disjoint simple closed curves on $\Sigma$ and
a \emph{pants decomposition} is a multicurve with $3g-3$ curves that decomposes $\Sigma$ into pants (3-holed 2-spheres).
Given a pants decomposition $P$  and a  trivalent graph  $\Gamma$ embedded on $\Sigma$ we will say that  $P$ and $\Gamma$ are \emph{dual} 
(or $\Gamma$ is dual to $P$) if each vertex of $\Gamma$ lies in exactly one pants of $P$, and each edge of  $\Gamma$ intersects a single curve of $P$ exactly once.
When we want to stress the fact that $\Gamma$ is dual to $P$, we will use the notation $\Gamma_P$.

Given a pants decomposition $P$ with dual graph $\Gamma$  let $E$ denote the set of edges of $\Gamma$.
We will use the notation $P=\{ \alpha_e \  |\  e\in E\}$, where $ \alpha_e$ denotes the curve in $P$ that intersects the edge $e$ of $\Gamma$. Following \cite{BHMV2},
we assume that $\Gamma$ is 
\emph{banded} (i.e.  thickened with a cyclic ordering at each vertex) 
and that $\s$ is obtained as a double of $\Gamma$.
 
Recall that for any integer $r\geq 3$,  we defined ${\mathcal C}_r=\{ 1, 2,\ldots, r-1\}$. For $\zeta_r=-e^{\frac{\pi i}{2r}}$, the $SU(2)$-Witten-Reshetikhin-Turaev TQFT
 \cite{RETu, Turaevbook, witten} assigns to $\Sigma$ a finite dimensional $\CC$-vector space $V_r(\Sigma)$ with
 a non-degenerate Hermitian pairing
$$\langle\ \  \rangle: V_r(\Sigma)\times V_r(\Sigma) \longrightarrow \CC.$$

\begin{definition} \label{def:admissible} An $r$-\emph{admissible coloring} of $\Gamma$
is a function ${\bf c}: E\longrightarrow {\mathcal C}_r$ such that 
for any triple of edges $e,f, g$ adjacent to the same vertex of $\Gamma$, we have
\begin{itemize}

\item [(i)]  ${c}_e+{c}_f+{c}_g<2r$ and ${c}_e+{c}_f+{ c}_g$ an odd number;  and
\item  [(ii)] ${c}_e<{c}_f+{ c}_g$, ${ c}_f<{c}_e+{ c}_g$ and ${c}_g<{ c}_e+{c}_f$.  
\end{itemize}
Let $U_r:=U_r(\Gamma)$ denote the set of all $r$-admissible colorings of $\Gamma$.
\end{definition}

\begin{remark}\label{twovsone}{\rm
We warn the reader that, compared to \cite{BHMV2}, we shifted all colors by $1,$ following the conventions of \cite{MP} for example. Hence, for instance, the color $2$ corresponds to the fundamental $U_q(sl_2)$-representation of dimension $2$.}
\end{remark}

By  \cite{BHMV2} the number of admissible colorings is independent of the choice of $P$  or $\Gamma$ and it is equal to the  dimension of $V_r(\Sigma)$:
Given ${\bf c}\in U_r$ the authors in \cite{BHMV2} give a construction that associates a vector $\phi_{{\bf c}}$ in  $V_r(\Sigma)$, such that 
the set ${\mathcal B}_P:={\mathcal B}^r_P=\{\phi_{{\bf c}} \ \ | \ \ {\bf c}\in U_r\}$ is basis of $V_r(\Sigma)$ that is orthonormal with respect to the Hermitian pairing. Throughout the paper we will refer to 
${\mathcal B}_P$ as the orthonormal basis corresponding to $P$ (or sometimes to $\Gamma$).
\vskip 0.05in 

Given a multicurve $\gamma$, \cite{BHMV2} assigns Hermitian operators  $T^{\gamma}_r:  V_r(\Sigma)\longrightarrow V_r(\Sigma),$ known as \emph{curve operators}:
The spaces $V_r(\Sigma)$ are certain quotients of a version of the Kauffman skein module $K(H,  \zeta_r)$, where $H$ is a handlebody with $\partial H=\Sigma$.  The
 product on $K(H,  \zeta_r)$, where for any  $w\in K(H,  \zeta_r)$ the skein  $\gamma . w$ is the result of stacking $\gamma$ above $w$, descents to the Hermitian  operator $T^{\gamma}_r$
 on  $V_r(\Sigma)$. We recall two key properties that we need here, and we refer the reader to \cite[Section 2]{RD} and references there in for more details.
 \begin{itemize}
\item [-] The map $K(\Sigma,  \zeta_r)\longrightarrow \mathrm{End}( V_r(\Sigma))$ where $\gamma \longrightarrow T^{\gamma}_r$ is an algebra morphism.
\vskip 0.04in

\item [-] Let $\gamma$ is a multicurve such that each curve in it is parallel to some  $\alpha_e\in P$, where $P$ is a pants decomposition of $\s$. Then, for any vector
$\phi_{{\bf c}}$ of  the basis corresponding to $P$ we have
\begin{equation} \label{eq:eigenvectors}
T^{\gamma}_r(\phi_{{\bf c}})=\underset{e\in E}{\prod} \left(-2\cos{\frac{\pi c_e}{r}} \right)^{n_e}     \phi_{{\bf c}},
\end{equation}
where $n_e\geq 0$ is the number of curves  in $\gamma$ that are parallel to $\alpha_e \in P$.

\end{itemize}

The large-$r$  behavior of curve operators for closed surfaces of genus at least two was studied  by the first author \cite{RD}. To recall the results from \cite{RD} that we will use in this paper, we need some preparation.

Fix a  decomposition $P$ with dual graph $\Gamma$ and  corresponding basis ${\mathcal B}_P=\{\phi_{{\bf c}} \ \ | \ \ {\bf c}\in U_r\}$.
Let $ {\RR}^{E}$ denote the set consisting of all $(3g-3)$-tuples, defined by functions  ${\bf \tau}: E \longrightarrow (0, 1)$. Consider the subset $U \subset {\RR}^{E}$ consisting of all $(3g-3)$-tuples, defined by functions such that  for any triple of edges $e,f, g$ adjacent to the same vertex of $\Gamma$ we have 
 \begin{itemize}
 \item[(i)] ${ {\tau}}_e+{ \tau}_f +\tau_g<2$; and
\item  [(ii)]  ${\tau}_e<{\tau}_f+{\tau}_g$, ${\tau}_f<{\tau}_e+{\tau}_g$ and ${\tau}_g<{\tau}_e+{\tau}_f$.
  \end{itemize}
 
By Definition \ref{def:admissible}, for every $r\geq 3$ 
and any admissible coloring ${\bf c}\in U_r,$  
we have $\tau:=\frac{{\bf c}}{r} \in U,$ where $\tau_e:=\frac{ c_e}{r}$, for all $e\in E$.

We recall that given  ${\bf c}\in U_r$ the article \cite{RD} defines a function
$\overline{\bf c}: H_1(\Sigma, \ZZ/{2\ZZ})\longrightarrow \{\pm 1\} $, as follows:  Let $L_{\bf c}$ be a multicurve on $\Sigma$ consisting of $c_e-1$ parallel strands along each edge $e$ of $\Gamma,$ joined at vertices without crossing. As $\Sigma$ is obtained by gluing together two copies of $\Gamma$
along corresponding boundary components, let $\tilde{L_{\bf c}}$ consist of two copies in $L_c,$ one in each copy of $\Gamma.$
For $\gamma$  a multicurve on $\Sigma $, we define $\overline{\bf c}(\gamma)=(-1)^{I (\tilde{L_{\bf c}},  \gamma)},$  where
$I (\tilde{L_{\bf c}},  \gamma)$ is the total geometric intersection number of $\gamma$ and $\tilde{L_{\bf c}}$. 

\begin{theorem}\label{thm:Asymptotic} \cite[Theorem 1.1]{RD} Fix $P$, $\Gamma$ and $U$ as above. Given a multicurve $\gamma$, there is an open subset $V_{\gamma}\subset U\times [0, 1]$ containing $U\times \lbrace 0 \rbrace,$ and unique analytic functions
$$G_{\bf k}^{\gamma}: V_{\gamma} \longrightarrow \CC, \ \ {\rm where} \  \ {\bf  k}: E\longrightarrow \ZZ, $$
 such that we have:
 \begin{enumerate}
\item For any ${\bf c}\in U_r$ such that $(\frac{\bf c}{r},\frac{1}{r}) \in V_{\gamma},$ we have $G_{\bf k}^{\gamma}(\frac{\bf c}{r},  \frac{ 1}{r})= \overline{\bf c}(\gamma)\langle T^{\gamma}_r( \phi_{{\bf c}}), \ \phi_{{\bf c}+{\bf k}}\rangle$ for any ${\bf c}\in U_r$.
 \item We have $G_{\bf k}^{\gamma}=0$ for any ${\bf  k}$ for which there is  $e\in E$ such that $|{\bf  k}(e)|>I(\alpha_e, \gamma)$ or  ${\bf  k}(e)\not\equiv I(\alpha_e, \gamma)(\rm{mod}\ 2)$.
 Here, $\alpha_e$ is the curve in $P$ that corresponds to $e$ and $I(\alpha_e, \gamma)$ is the geometric intersection number of $\alpha_e, \gamma$.
  \end{enumerate}
\end{theorem}
\begin{remark} {\rm Strictly speaking, in \cite{RD} it is only proven that such analytic functions $G_{\bf k}^{\gamma}$ exist and uniqueness is not discussed.
Note, however, that if an analytic function $G_k^{\gamma}$ on $V_{\gamma}$ satisfies part (1) of Theorem \ref{thm:Asymptotic}, then it is determined on all elements $(\frac{\bf c}{r},\frac{1}{r}) \in V_{\gamma}$ where ${\bf c}\in U_r.$ For $x \in U,$ we can choose a sequence of $r$-admissible colorings ${\bf c}_r$ such that $(\frac{{\bf c}_r}{r},\frac{1}{r})\underset{r\rightarrow \infty}{\longrightarrow} (x,0).$ From this we deduce  that $G_{\bf k}^{\gamma}$ is uniquely determined on $U\times \lbrace 0 \rbrace,$ and also, looking at the expansion of $\langle T_r^{\gamma}(\phi_{{\bf c}_r}) ,\phi_{{\bf c}_r+\bf{k}} \rangle$ at any order in $\frac{1}{r},$  that all derivatives of $G_{\bf k}^{\gamma}$ with respect to the second variable are determined on $U \times \lbrace 0 \rbrace,$ which implies that $G_{\bf k}^{\gamma}$ is uniquely determined on $V_{\gamma}$. }
\end{remark}

\begin{definition} \label{def:TQFTfunctions}Given $P$ a pants decompositions of $\Sigma$ and $\gamma$ a multicurve, consider the operators $T^{\gamma}_r$ and their expansion in the sense of Theorem \ref{thm:Asymptotic} with respect to the bases corresponding to $P$.
Define $n(\gamma, P)$ to be the number of ${\bf k}\neq {\bf 0}$, such that $G_{\bf k}^{\gamma}\neq {\bf 0}.$  We will call $n(\gamma, P)$ the \emph{quantum intersection number} of 
$\gamma$ and $P$.

Also if $Q$ is another pants decomposition, which we can view as a multicurve, define $n(Q,P)$ also as the number of ${\bf k}\neq {\bf 0}$ such that $G_{\bf k}^Q\neq {\bf 0}.$
\end{definition}

Let us stress that the definition of $n(Q,P)$ is not symmetric in $P,Q$. The pants decomposition $Q$ plays the role of a multicurve while the pants decomposition $P$ provides a basis of $V_r(\Sigma).$ 
We also stress that we are  restricting ourselves ${\bf k}\neq {\bf 0}$ so we don't count the diagonal terms $G_{\bf 0}^{\gamma}$ in Definition \ref{def:TQFTfunctions}.

\subsection{Two basic properties}
Here we derive two basic properties of the quantum intersection number.
We begin with the following lemma.

\begin{lemma} \label{lemma:split} Fix a pants decomposition $P$ of $\s$.  Any multicurve $\gamma$,  splits into multicurves $\gamma'$ and $\gamma''$  (some of which may be empty), so that each  curve in $\gamma''$ is parallel to some curve in $P$.
We have $n(\gamma, P)=n(\gamma', P)$.

\end{lemma}
\begin{proof} By the  properties of curve operators we recalled earlier, we can write $T^{\gamma}_r$ as  a product $T^{\gamma}_r=T^{\gamma'}_r T^{\gamma''}_r$, where
the matrix representing $T^{\gamma''}_r$ with respect to the basis corresponding to $P$ is diagonal with non-zero eigenvalues (see Equation \eqref{eq:eigenvectors}).
Thus, for any $r\geq 3$, the number of non-zero entries in the matrices representing  $T^{\gamma}_r$ and $T^{\gamma'}_r $ is the same.
\end{proof}

The mapping class group
$\M$  acts on the set of multicurves of $\s$ and on $\CV$. The next result shows that the the quantum intersection number 
is invariant under this action.

\begin{proposition}\label{prop:MCGaction} For any $\phi \in \mathrm{Mod}(\Sigma)$,  we have $n(\phi(\gamma), \phi(P))=n(\gamma,P),$ for all multicurves $\gamma$ and $P\in \CV$.
\end{proposition}
\begin{proof} We recall from \cite{BHMV2} that for any $r>2$  we have a projective representation 

$$\rho_{r} : \mathrm{Mod}(\Sigma) \rightarrow \mathbb{P}\mathrm{End}( V_r(\Sigma)).$$
It follows from the definitions that for any $\phi \in \M$, and any multicurve $\gamma$, we have
$$\rho_r(\phi)T_r^{\gamma}\rho_r(\phi)^{-1}=T_r^{\phi(\gamma)}.$$

Hence, in particular,  if $P$ is a pants decomposition, then $\rho_r(\phi)T_r^{P}\rho_r(\phi)^{-1}=T_r^{\phi(P)}.$ Moreover, $\rho_r(\phi)$ sends the  basis  ${\mathcal B}_P$ of 
$V_r(\Sigma)$
associated to $P$,  to the basis ${\mathcal B}_{\phi(P)}$  associated to $\phi(P)$, up to scalars that are roots of unity (up to \emph{phases}).
By Equation \eqref{eq:eigenvectors}, the basis ${\mathcal B}_{P}$ is an orthonormal basis of eigenvectors of the curve operators $T_r^{\alpha_e}$,
for  $\alpha_e\in P$,  and the corresponding eigenvectors are non-zero.
Moreover, such a basis is unique up to phases. Therefore, for all $r\geq 3,$ the coefficients of $T_r^{\gamma}$ in ${\mathcal B}_{P}$ are the same (up to phases) as the coefficients of $T_r^{\phi(\gamma)}$ in the basis  ${\mathcal B}_{\phi(P)}$. This implies that $n(\gamma,P)=n(\phi(\gamma), \phi(P)).$
\end{proof}

\subsection{Quantum and geometric intersection number}
The main technical result of the paper is the following theorem.
\begin{theorem}\label{thm:inequalitiesgeneral} Given $P$ and  $\gamma$  a pants decomposition  and a multicurve on $\s$, let  $I(\gamma,P)$ denote the total intersection number of $\gamma$ and $P$. We have 
$$
\frac{I(\gamma,P)}{3g-3}\leq  n(\gamma,P)  \leq (I(\gamma,P)+1)^{3g-3}-1,
$$
where $n(\gamma,P)$ denotes the quantum intersection number of $\gamma$ with respect to $P$.
\end{theorem}

The upper inequality in Theorem \ref{thm:inequalitiesgeneral}, which a direct consequence of Theorem \ref{thm:Asymptotic}, is derived in the next lemma.

\begin{lemma}\label{lemma:upper} We have $n(\gamma,P)  \leq (I(\gamma,P)+1)^{3g-3}-1$,
for all $P\in \CV$ and  multicurves $\gamma$ on $\s$.
\end{lemma}
\begin{proof} If $I(\gamma,P)=0$, then
both sides of the inequality are zero and there is nothing to prove. 

Suppose that $I(\alpha_e, \gamma)\neq 0$ for some $e\in E$, and 
that ${\bf  k}: E\longrightarrow \ZZ$ is a function such that  $G_{\bf k}^{\gamma}\neq {\bf 0}$ and ${\bf  k}(e)\neq 0$.

Then,
by  part (2) of Theorem \ref{thm:Asymptotic}, 
 there are at most $I(\alpha_e, \gamma)+1$ choices 
for ${\bf  k}(e)$. Thus, and since in Definition \ref{def:TQFTfunctions}
we only count there non-diagonal terms $G_{\bf k}^{\gamma}\neq {\bf 0},$ we have
$$n(\gamma,P)  \leq \underset{e\in E}\prod (I(\alpha_e,P)+1)-1 \leq(I(\gamma,P)+1)^{3g-3}-1.$$
\end{proof}

Note that Theorem \ref{thm:inequalities} is a special case for Theorem  \ref{thm:inequalitiesgeneral}  where the multicurve $\gamma$ is also a pants decomposition $Q$.

\begin{remark}\label{rem:CM} {\rm Given  a pants decomposition $P=\{\alpha_e\  |\  e\in E\}$  and a multicurve
$\gamma$, consider the set of maps
${\bf  m}: E\longrightarrow \ZZ,$
such that for any $e\in E$, we have ${\bf  m}(e)\in \{ \pm I(\gamma, \alpha_e)\}$.
If $I(\gamma,P)\neq 0$, let $m_{\gamma}\leq 3g-3$ denote the number of edges $e\in E$ such that  $I(\gamma, \alpha_e)\neq 0$.
By \cite[Theorem 1.3] {RD} and a result of  Charles and March\'e  \cite[Theorem 5.1]{CM} it follows that for each ${\bf m}\neq {\bf 0}$
we have $G_{\bf m}^{\gamma}\neq 0$. This gives that
$ n(\gamma,P)\geq 2^{m_{\gamma}}$, which was the only lower bound of the quantum intersection number known before Theorem \ref{thm:inequalitiesgeneral}.}
\end{remark}

We now briefly describe our approach to proving the lower inequality of Theorem \ref{thm:inequalitiesgeneral}: Let $P$ be a pants decomposition of $\Sigma$ with dual
graph $\Gamma$. Given a  multicurve $\gamma \neq P$, define
\begin{equation}
M(\gamma, P):=\max_{\alpha_e\in P}\{I(\gamma, \alpha_e)\}.
\label{eq:maxint}
\end{equation}
We can pick a curve  $\alpha_{e_0}$, dual to an edge $e_0\in \Gamma$ such that $M( \gamma, P)=I(\gamma,  \alpha_{e_0})$.
For any integer  $\abs{\sigma} \leq I(\gamma,\alpha_{e_0}) $ consider the function ${\bf  k}_{\sigma}: E\longrightarrow \ZZ,$
defined by $${\bf  k}_{\sigma}(e)=\begin{cases} \sigma \ \ \ \ \  \ \ \  \ \  {\rm if}\ \  e=e_0,&\\
I(\gamma, \alpha_{e})\ \ {\rm if} \ \  e\neq e_0.\end{cases},$$
where $\alpha_e \in P$ is the curve dual to the edge $e\in \Gamma$.
We will show the following:
\begin{theorem}\label{thm:lower} 
For each  integer $\sigma \neq 0$, with $ \abs{\sigma}\leq M(\gamma, P)$ and $\sigma\equiv M(\gamma, P)(\rm{mod}\ 2)$, we have $G_{{\bf k}_{\sigma}}^{\gamma}\neq {\bf 0}.$ 
\end{theorem}
\vskip 0.04in

By Equation \eqref{eq:maxint},  $I(\gamma, P)\leq ({3g-3}) M(\gamma, P)$ and  by
Theorem \ref{thm:lower} and Definition \ref{def:TQFTfunctions}, $M(\gamma, P)\leq n(\gamma, P)$. Hence, the lower inequality  of Theorem \ref{thm:inequalitiesgeneral} 
follows once we have Theorem  \ref{thm:lower}.  The proof of Theorem  \ref{thm:lower} is given  in Section \ref{sec:five}.

\subsection{Elementary moves and quantum intersection number}
Let $P$ be a pants decomposition of a surface $\s$ of genus $g>1$. Consider a curve $\alpha_e\in P$ bordering  two pairs of pants giving a four-holed 2-sphere subsurface
of $\Sigma_{0,4}$ of $\s$.
\begin{definition}\label{def:emoves}
An $A$-move  replaces  $\alpha_e\in P$ with an essential curve on $\Sigma_{0,4}$ that intersects $\alpha_e$ at two points. The effect of an $A$-move  on  the dual graph
$\Gamma$ is \emph{elementary shift}.
See top row of Figure \ref{fig:ASmoves}.

An $S$-move  replaces a curve $\alpha_e\in P$ that lies in the interior of an  one-holed torus subsurface $\Sigma_{1,1} \subset \s$, by a curve that lies on $\Sigma_{1,1}$
and intersects $\alpha_e$   once. The $S$-move leaves the graph $\Gamma$ unchanged.
See  bottom row of Figure \ref{fig:ASmoves}. 

Shortly, we will refer to an $A$ or $S$-move as an \emph{elementary move}.
\end{definition}
Note that an elementary move on a pants decomposition $P$ produces another pants  decomposition of $\s$.

\begin{figure}[h]
\centering
\def \svgwidth{.65\columnwidth}
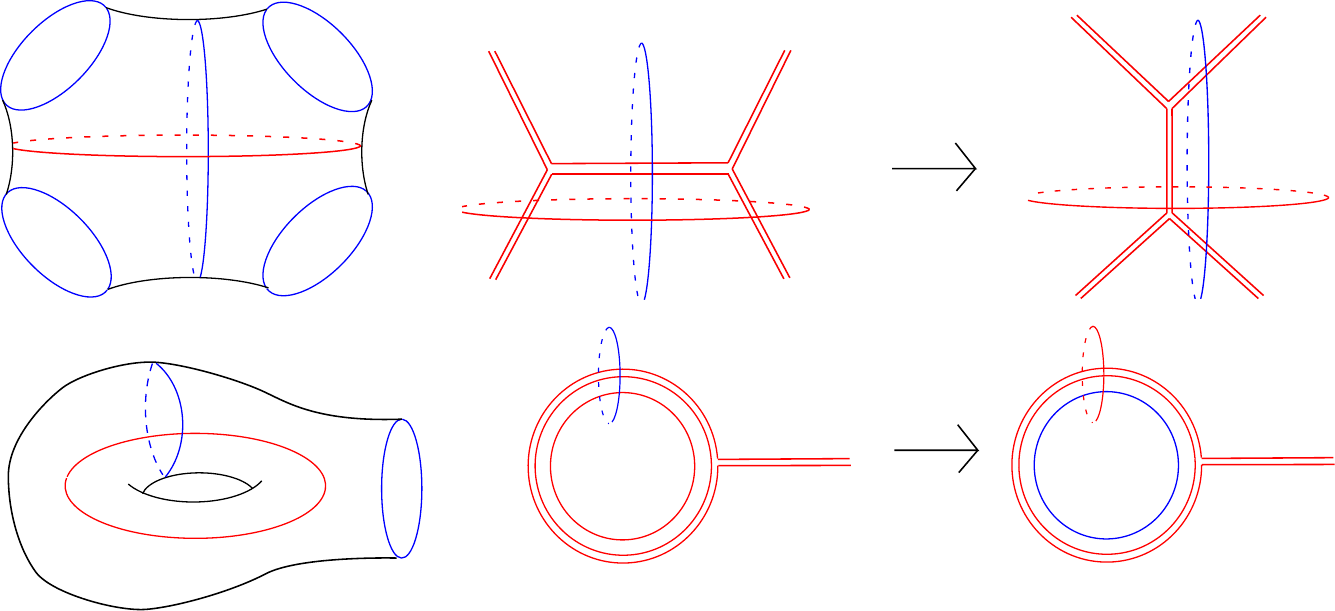
\caption{Elementary moves of pants decompositions and  dual graphs. The portion of the dual graph affected by the elementary move is indicated by the red double lines
(grey double lines in black-and-white/grey scale).
In both moves the closed curves shown in the interior of the surface shown are exchanged. }\label{fig:ASmoves}
\end{figure}

\begin{proposition}\label{prop:ASmovecomputation} If $P, Q$ are pants decompositions that differ by an elementary move, then
$n(P,Q)=n(Q,P)=2$.
\end{proposition}
\begin{proof}The proof is given in \cite[Proposition 4.2]{RD}. Alternatively, the proposition follows from \cite[Theorem 1.3]{RD} and  \cite[Theorem 5.1]{CM}
as mentioned in Remark \ref{rem:CM}:  Let $\alpha_{e_0}$ be the curve of $P$ that is changed using the $A$- or $S$-move. 
Since $I(Q, \alpha_e)=0$ for any $e\neq e_0$, and $I(Q, \alpha_{e_0})\in \{1,2\}$, Theorem \ref{thm:Asymptotic} and Remark \ref{rem:CM} imply the result.
\end{proof}

The proof of Theorem \ref{thm:lower} will occupy the next three sections. By Lemma \ref{lemma:split}, during the entire course of the proof, we may and will assume that the multicurve $\gamma$
does not have any component isotopic to a curve in $P$. That is  $I(\gamma, \alpha_e)\neq 0$ for each $ \alpha_e\in P$.

With  the statement of Theorem \ref{thm:inequalitiesgeneral} at hand, the reader who is eager to move to the proofs of Theorems \ref{thm:QI} and
\ref{thm:volume}    could move to Section \ref{sec:TSconnections}  without loss of continuity.


\section{ Dehn-Thurston position and  Fusion rules} 
\label{sec:fusions}

In this section we  review the basics of Dehn-Thurston position for curves on surfaces and the fusion theory for computing
coefficients of curve operators, both adapted in the setting that will be used for the
proof of Theorem \ref{thm:lower}.

\subsection{Dehn-Thurston position and swift number of arcs} Let us fix a pants decomposition 
$P=\{\alpha_e \ | \ e \in E\}$ of $\s$ with  dual graph $\Gamma$, and
with set of edges $E$. We replace each curve $\alpha_e$ in $P$ by two parallel copies which we denote by $\alpha_e,\alpha_e'$,  and we let $P':=\{\alpha'_e \ | \ e \in E\}.$

Therefore we get a decomposition of $\Sigma$ into $2g-2$ pairs of pants and $3g-3$ annuli. 

\begin{definition} \label{decomposition} The set of curves $P\cup P'$ will be called a \emph{decomposition system} of $\Sigma.$
We will refer to the components  of $\Sigma\setminus (P\cup P')$ as the pieces of the system $P\cup P'$.
\end{definition}

\begin{figure}[!h]
\centering
\def \svgwidth{.75\columnwidth}
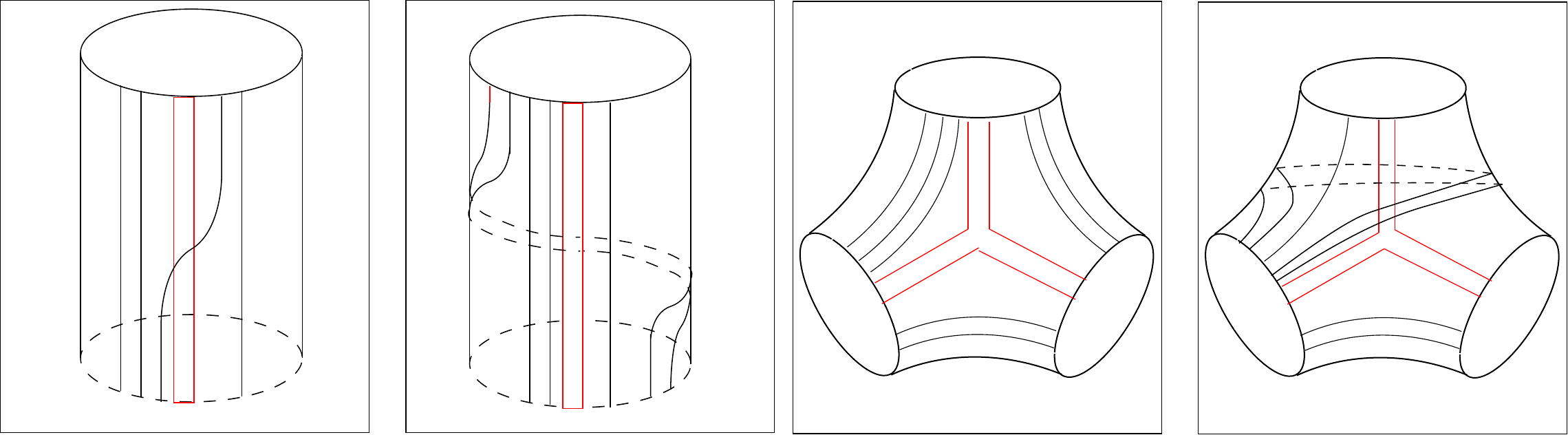
\caption{Patterns of a  multicurve in Dehn-Thurston position on  pieces of a decomposition system.}
\label{fig:DehnThurston}
\end{figure}
A multicurve $\gamma$ on $\Sigma$ is said to be in \emph{Dehn-Thurston position} with respect to $P\cup P'$, if we have the following: 
\begin{itemize}
\item[-] The intersection  of  $\gamma$ with 
each piece of $P\cup P'$  consists  of properly embedded arcs so that the number of intersection points  with each curve $\alpha_e$ (or $\alpha_e'$) is  $I(\gamma,\alpha_e).$
\item[-] On the  pants pieces   of $P\cup P'$ there are two patterns for arcs  of $\gamma$: Arcs run between different boundary components of
a  pants piece, or they have both endpoints on the same boundary component of the piece. Arcs of the later kind will be referred as \emph{come-back patterns}.
See Figure \ref{fig:DehnThurston}, where the portion of $\Gamma$ on the decomposition pieces illustrated is shown in red double line (shown as  grey in white-and-black/grey scale).

\item[-] On the annuli pieces $\gamma $ is determined, up to  \emph{ fractional Dehn twists}, by the number of intersection points with the boundary of the annuli. Each arc is assigned the \emph{swift number}, which is defined in Definition \ref{def:swift} below.
Examples are given in  Figure \ref{fig:DehnThurston}.
 \end{itemize}

Given a decomposition system of $\Sigma$ any multicurve  $\gamma$ can be isotoped into Dehn-Thurston position \cite{FLP}.

To give the precise definition of swift numbers we need some conventions  on annuli pieces of $P\cup P'$:
We consider such an annulus $A=S^1\times [0, 1]$ embedded in ${\mathbb R}^3$ so that, for $i=0,1$, $\partial A^{i}:=S^1\times \{i\}$
 is identified with the unit circle centered at the origin on the plane $z=i$. Thus, traveling in the counter clockwise direction we can talk of the first, second, third and fourth quarter of $\partial A^{i}$.

 Given $\gamma$ in Dehn-Thurston position, we  will assume that the points of $\gamma \cap \partial A^{i}$ occur in the third and fourth quarter of $ \partial A^{i}$.
Recall that $\s$ is obtained by gluing $\Gamma$ to a second copy of itself  $\Gamma'$.
We will assume that $e_S:=\Gamma \cap A$ (resp $e_N:=\Gamma' \cap A $)
is a straight segment on $A$
that runs from the south pole (resp. north pole)
of  $ \partial A^{1}$ to the south pole (resp. north pole) of  $ \partial A^{0}$. In Figure \ref{fig:DehnThurston}, $e_S$ is shown in red double lines (grey in white-and-black/grey scale) while $e_N$ is omitted.

\begin{definition}\label{def:swift} 
With the notation and setting as above, let $y$ be an arc $\gamma \cap A$  that runs from $ \partial A^{1}$ to $ \partial A^{0}$
and let $|y\cap (e_N\cup e_S)|$ denote the geometric intersection number of $y$ with the arcs $e_N,e_S$.  The
\emph{swift number} of $y$, denoted by $t_y$ 
is defined as follows:   
\begin{itemize}
\vskip 0.04in

\item [-] We have $t_y=0$,  if $|y\cap (e_N\cup e_S)|=0$.
\vskip 0.04in

\item[-] We have $t_y:=|y\cap (e_N\cup e_S)|$,
 if  either
 \vskip 0.03in
 \begin{enumerate}[(i)]
  \item $y\cap \partial A^{1}$ is in the third quarter of $ \partial A^{1}$ and the first point of  the intersection of $y$ with $e_N\cup e_S$ is on $e_N$; or
   \vskip 0.03in
 \item $t_y:=|y\cap (e_N\cup e_S)|$,
 if $y\cap \partial A^{1}$ is in the fourth quarter of $ \partial A^{1}$ and the first point of  the intersection of $y$ with $e_N\cup e_S$ is on $e_S$.
 \end{enumerate}
 
 \vskip 0.04in
 
 \item [-]We have $t_y:=-|y\cap (e_N\cup e_S)|$,
 if either
  \vskip 0.03in
 \begin{enumerate}[(i)]
 \item $y\cap \partial A^{1}$ is in the third quarter of $ \partial A^{1}$ and the first point of  the intersection of $y$ with $e_N\cup e_S$ is on $e_S$;  or
  \vskip 0.03in
 \item if $y\cap \partial A^{1}$ is in the fourth quarter of $ \partial A^{1}$ and the first point of  the intersection of $y$ with $e_N\cup e_S$ is on $e_N$.
 \end{enumerate}

\end{itemize}
\end{definition}
To illustrate Definition \ref{def:swift}, we discuss the swift numbers of the arcs appearing on the two cylinder pieces of Figure \ref{fig:DehnThurston}:
In the first from the left panel of the figure, we have four arcs, three of which have swift number 0 and one has swift number 1.
In the second panel
of Figure \ref{fig:DehnThurston}, we have three arcs of swift number 0 and two arcs of swift number -1.

\subsection{Fusion rules at a limit} Given a multicurve $\gamma$ in Dehn-Thurston position with respect to a decomposition system $P\cup P'$, we can use
a version of the fusion rules of Masbaum-Vogel \cite{MV94}  to compute  the 
coefficient functions $G_{\bf k}^{\gamma}$ of the operators $T^{\gamma}_r$. The adaptation of the fusion rules of \cite{MV94} 
that applies to our setting, with the orthonormal bases of the  TQFT spaces corresponding to $P$ (described in Section \ref{sec:two}), was given in \cite{RD}.
For the convenience of the reader we recall these fusion rules in Figure \ref{fig:fusionRules} in Appendix \ref{sec:appendix} of the paper. For more details the reader  is referred to \cite[Section 4]{RD}.

In order to simplify the search for non-vanishing coefficients $G_k^{\gamma},$ for the proof of Theorem \ref{thm:lower},  it is convenient for us consider their limit when $r\rightarrow\infty,$  $c_e=a$ for all edges $e$, and $\frac{a}{r}\rightarrow \theta$ where $\theta \in [0,1].$
 This is equivalent to just computing the evaluation $G_k^{\gamma}(\theta,0),$ and thus clearly a lower bound for the number of non-vanishing coefficients $G_k^{\gamma}.$ 
The advantage of working with this limit is that we work with a simplified set of fusion rules by replacing the general fusion rules given in Figure \ref{fig:fusionRules} by their limit
when $r\rightarrow\infty$, we have   $c_e=a$ for every edge $e$ of $\Gamma$,  and $\frac{a}{r}\rightarrow \theta$ where $\theta \in [0,1].$ 
These leads to simplified coefficients
as shown  Figure \ref{fig:fusionRules2}. In the figures, the colors of thick red edges are labelled by letters. The  thin black edges are colored by $2$.
That is they are colored by the fundamental $U_q(sl_2)$-representation of dimension $2$, which, as explained in Remark \ref{twovsone}, corresponds to the color $1$ of
\cite{BHMV2}.

\begin{figure}[!h]
\centering
\def \svgwidth{.70\columnwidth}
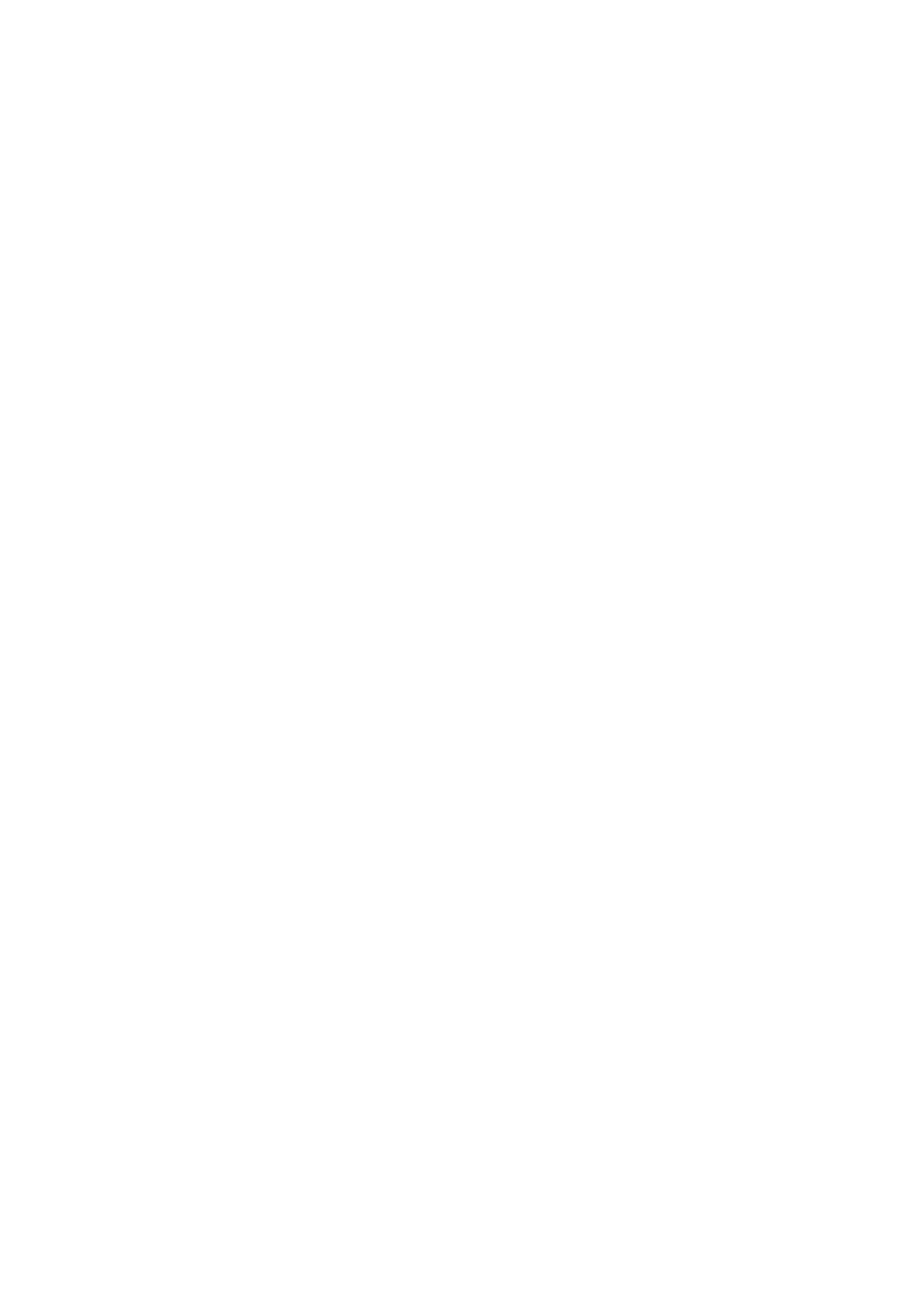
\caption{The limit of fusion rules when $c_e=a,$ and $r\rightarrow \infty, \frac{a}{r}\rightarrow \theta.$ Here we let $\langle \theta \rangle=\sin(\pi \theta)$.
The red thick lines (shown in grey in black-and-white/grey scale)
indicate portions of the graph $\Gamma$.
The black thiner lines indicate  arcs of $\gamma$ and the color on these portions is $2$. }
\label{fig:fusionRules2}
\end{figure}
\begin{remark}\label{rem:sign} {\rm In the above limiting setting,  the fusion coefficients only depend of the type of the fusion rule used (i.e. strand, positive/negative half-twist, positive/negative/mixed triangle, positive/negative bigon) and not on  the colors of edges.
We also note that the $r$-admissibility condition (i) of Definition \ref{def:admissible} implies that $a$ is odd.  As a result the signs $(-1)^{a+1}$  for the fusion coefficients
of twist rules in the second row of Figure \ref{fig:fusionRules2}  do not appear in our computations.}
\end{remark}


\section{Fusion computations at a limit}
\label{sec:fusion_computations}

\subsection{A state sum formula} With the notation and terminology of Section  \ref{sec:fusions}, let us fix a decomposition system $P\cup P'=\{\alpha_e, \alpha'_e \ | \ e\in E\}$ of $\Sigma$, with dual graph $\Gamma$.
Let $\gamma$ be a multicurve on $\Sigma$ that is in
Dehn-Thurston position with respect to $P\cup P'$.

\begin{definition}\label{def:totalshift} A  \emph{total color shift} on $\Gamma$ is an $E$-tuple $\sigma=(\sigma_e| e\in E)$, defined by a function $\sigma: E\longrightarrow \ZZ$,
 such that
$\abs{\sigma_e} \leq I(\gamma, {\alpha}_e)$.
\end{definition}

Recall that we are considering the functions
$G_k^{\gamma}$  of Theorem \ref{thm:Asymptotic},
at the limit when $r\rightarrow\infty,$  $c_e=a$ for all edges $e$, and $\frac{a}{r}\rightarrow \theta$ where $\theta \in [0,1].$ 
 Let $\phi_a$ be the TQFT vector corresponding to the coloring that assigns $a$ to each edge of $\Gamma$.
 Given a total  color shift $\sigma$, we will use $C(\sigma)$ to denote the limit of the coefficients of $T_r^{\gamma}\phi_a$ along $\phi_{a+\sigma}.$
 We write

\begin{equation}
\label{eq:coeffcient}
T_r^{\gamma} \phi_a= \sum_{\sigma} C(\sigma) \ \phi_{a+\sigma} \ +o(1),\end{equation}
where $\sigma$ ranges over all possible total color shifts, and $o(1)$ denotes a quantity that converges to $0$ as $r\rightarrow \infty$.

Our next goal is to give certain state sum expressions for the coefficients $C(\sigma)$.
Before we are able to do this we need some preparation: Let $X$ be a pants piece of the decomposition system $P\cup P'$.
Note that $X$ contains a $Y$-shaped portion of $\Gamma$, which we will denote by $Y$. 
Now $X$ can be constructed by gluing two copies  of $Y$, say $Y$ and $Y'$ along two of their sides.
 If $z$ is a come-back pattern of $\gamma$ on $X$, then $z$ intersects each of $Y, Y'$ exactly once.

 \begin{definition}\label{def:crossings} With the notation and setting as above, we will call the intersection point of $z$ with $ Y$, \emph{the overcrossing } of the come-back pattern.
 We will call the intersection point of $z$ with  $ Y'$, \emph{the undercrossing} of the come-back pattern.
 \end{definition}
 
 The multicurve $\gamma$ gives rise to a set of arcs $B$ on
  $ \Sigma$, such that for any  arc $y\in B$, with endpoints $p_1, p_2$, we have one of the following:
  \vskip 0.02in
 \begin{itemize}
 \item[-] $y$ lies on a pants piece of the decomposition system and  $p_1,p_2$ are on $P\cap \gamma$;
 \vskip 0.02in

 \item[-] $y$ lies on  a  pants piece and $p_1$ is  on $P\cap \gamma$ while $p_2$ is at an overcrossing or undercrossing of a come-back pattern of $\gamma$;
 \vskip 0.02in

 \item[-] $y$ lies on an annulus and $p_1$ is on $\alpha_e \cap \gamma$, for some $\alpha_e\in P$, while $p_2$ is on  $\alpha'_e \cap \gamma.$
 \end{itemize}
 We will refer to $B$ as an \emph{arc system} for $\gamma$ with respect to  $P\cup P'$.
  Let $B_{\rm{end}}$ denote the set of all the endpoints of the arcs in $B$.
 
 \begin{definition}\label{def:states} (\emph{States}) Let $P\cup P'$, $\gamma$, $B$ and $B_{\rm{end}}$  be fixed as above.
 Given   a total color shift $\sigma=(\sigma_e | e\in E)$, a \emph{state} on $B$ is a function
$s: B_{\rm{end}}\longmapsto \lbrace \pm 1 \rbrace,$
 such that:
 \begin{enumerate}
 \item[(a)] For any $e\in E,$ 
\begin{equation}\label{sumcurve}
\underset{p\in \gamma \cap \alpha_e}{\sum} s(p)=\underset{p\in \gamma \cap \alpha_e'}{\sum} s(p)=\sigma_e.
\end{equation}
 \item[(b)] If $p,q$ are endpoints corresponding to the overcrossing and undercrossing of the same come-back pattern of $\gamma$, then $s(p)+s(q)=0.$
  \end{enumerate}
 \end{definition}

For fixed $\sigma$, we will use $\mathcal{S}_{\sigma}$ to denote the set of all states on $B$ corresponding to $\sigma$.
 We compute the coefficients $C(\sigma)$ using the fusion rules of Figure \ref{fig:fusionRules2}. For this we need two auxiliary lemmas that we now state.
 The proofs of the lemmas are given in Section \ref{section:auxiliary}.
  \vskip 0.04in

\begin{lemma}\label{lemma:sliding}{{\rm (}\emph{Sliding  Lemma}{\rm )}} With the conventions of Figure \ref{fig:fusionRules2} and for any $\varepsilon,\mu\in \lbrace \pm 1 \rbrace$, we have
\begin{center}
\def \svgwidth{.7\columnwidth}
\begingroup%
  \makeatletter%
  \providecommand\color[2][]{%
    \errmessage{(Inkscape) Color is used for the text in Inkscape, but the package 'color.sty' is not loaded}%
    \renewcommand\color[2][]{}%
  }%
  \providecommand\transparent[1]{%
    \errmessage{(Inkscape) Transparency is used (non-zero) for the text in Inkscape, but the package 'transparent.sty' is not loaded}%
    \renewcommand\transparent[1]{}%
  }%
  \providecommand\rotatebox[2]{#2}%
  \newcommand*\fsize{\dimexpr\f@size pt\relax}%
  \newcommand*\lineheight[1]{\fontsize{\fsize}{#1\fsize}\selectfont}%
  \ifx\svgwidth\undefined%
    \setlength{\unitlength}{670.82627664bp}%
    \ifx\svgscale\undefined%
      \relax%
    \else%
      \setlength{\unitlength}{\unitlength * \real{\svgscale}}%
    \fi%
  \else%
    \setlength{\unitlength}{\svgwidth}%
  \fi%
  \global\let\svgwidth\undefined%
  \global\let\svgscale\undefined%
  \makeatother%
  \begin{picture}(1,0.13886528)%
    \lineheight{1}%
    \setlength\tabcolsep{0pt}%
    \put(0,0){\includegraphics[width=\unitlength,page=1]{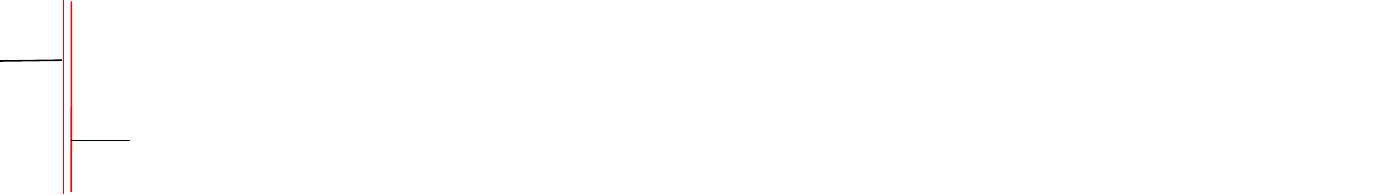}}%
    \put(0.05910307,0.00896349){\color[rgb]{0,0,0}\makebox(0,0)[lt]{\lineheight{1.25}\smash{\begin{tabular}[t]{l}$a$\end{tabular}}}}%
    \put(0.05910308,0.05208727){\color[rgb]{0,0,0}\makebox(0,0)[lt]{\lineheight{1.25}\smash{\begin{tabular}[t]{l}$a+\varepsilon$\end{tabular}}}}%
    \put(0.05730626,0.1101846){\color[rgb]{0,0,0}\makebox(0,0)[lt]{\lineheight{1.25}\smash{\begin{tabular}[t]{l}$a+\varepsilon+\mu$\end{tabular}}}}%
    \put(0.16230791,0.03186028){\color[rgb]{0,0,0}\makebox(0,0)[lt]{\lineheight{1.25}\smash{\begin{tabular}[t]{l}$=$\end{tabular}}}}%
    \put(0,0){\includegraphics[width=\unitlength,page=2]{SlidingA.pdf}}%
    \put(0.26883057,0.01049215){\color[rgb]{0,0,0}\makebox(0,0)[lt]{\lineheight{1.25}\smash{\begin{tabular}[t]{l}$a$\end{tabular}}}}%
    \put(0.26725439,0.05409531){\color[rgb]{0,0,0}\makebox(0,0)[lt]{\lineheight{1.25}\smash{\begin{tabular}[t]{l}$a+\mu$\end{tabular}}}}%
    \put(0.2668593,0.11159369){\color[rgb]{0,0,0}\makebox(0,0)[lt]{\lineheight{1.25}\smash{\begin{tabular}[t]{l}$a+\varepsilon+\mu$\end{tabular}}}}%
    \put(0,0){\includegraphics[width=\unitlength,page=3]{SlidingA.pdf}}%
    \put(0.54761938,0.00858486){\color[rgb]{0,0,0}\makebox(0,0)[lt]{\lineheight{1.25}\smash{\begin{tabular}[t]{l}$a$\end{tabular}}}}%
    \put(0.54761944,0.05170864){\color[rgb]{0,0,0}\makebox(0,0)[lt]{\lineheight{1.25}\smash{\begin{tabular}[t]{l}$a+\varepsilon$\end{tabular}}}}%
    \put(0.65760054,0.03317571){\color[rgb]{0,0,0}\makebox(0,0)[lt]{\lineheight{1.25}\smash{\begin{tabular}[t]{l}$=-$\end{tabular}}}}%
    \put(0,0){\includegraphics[width=\unitlength,page=4]{SlidingA.pdf}}%
    \put(0.54794228,0.11296386){\color[rgb]{0,0,0}\makebox(0,0)[lt]{\lineheight{1.25}\smash{\begin{tabular}[t]{l}$a+\varepsilon+\mu$\end{tabular}}}}%
    \put(0,0){\includegraphics[width=\unitlength,page=5]{SlidingA.pdf}}%
    \put(0.83671225,0.00858486){\color[rgb]{0,0,0}\makebox(0,0)[lt]{\lineheight{1.25}\smash{\begin{tabular}[t]{l}$a$\end{tabular}}}}%
    \put(0.83671232,0.05170864){\color[rgb]{0,0,0}\makebox(0,0)[lt]{\lineheight{1.25}\smash{\begin{tabular}[t]{l}$a+\mu$\end{tabular}}}}%
    \put(0.83703522,0.11296387){\color[rgb]{0,0,0}\makebox(0,0)[lt]{\lineheight{1.25}\smash{\begin{tabular}[t]{l}$a+\varepsilon+\mu$\end{tabular}}}}%
    \put(0,0){\includegraphics[width=\unitlength,page=6]{SlidingA.pdf}}%
    \put(0.4009281,0.05778418){\makebox(0,0)[lt]{\lineheight{1.25}\smash{\begin{tabular}[t]{l}and\end{tabular}}}}%
  \end{picture}%
\endgroup%

\end{center} 
\end{lemma}

The second auxiliary lemma shows how to simplify some arcs in annuli pieces of the decomposition system. Those arcs are the type of arcs that remain in annuli pieces after we applied the strand rules.

 \vskip 0.04in

\begin{lemma}\label{lemma:singleArcAnnuli} With the conventions of Figure \ref{fig:fusionRules2} and for any $\varepsilon,\mu \in \lbrace \pm 1 \rbrace$,  we have the following identity
\begin{center}
\def \svgwidth{.45\columnwidth}
\begingroup%
  \makeatletter%
  \providecommand\color[2][]{%
    \errmessage{(Inkscape) Color is used for the text in Inkscape, but the package 'color.sty' is not loaded}%
    \renewcommand\color[2][]{}%
  }%
  \providecommand\transparent[1]{%
    \errmessage{(Inkscape) Transparency is used (non-zero) for the text in Inkscape, but the package 'transparent.sty' is not loaded}%
    \renewcommand\transparent[1]{}%
  }%
  \providecommand\rotatebox[2]{#2}%
  \newcommand*\fsize{\dimexpr\f@size pt\relax}%
  \newcommand*\lineheight[1]{\fontsize{\fsize}{#1\fsize}\selectfont}%
  \ifx\svgwidth\undefined%
    \setlength{\unitlength}{225.31912255bp}%
    \ifx\svgscale\undefined%
      \relax%
    \else%
      \setlength{\unitlength}{\unitlength * \real{\svgscale}}%
    \fi%
  \else%
    \setlength{\unitlength}{\svgwidth}%
  \fi%
  \global\let\svgwidth\undefined%
  \global\let\svgscale\undefined%
  \makeatother%
  \begin{picture}(1,0.29950518)%
    \lineheight{1}%
    \setlength\tabcolsep{0pt}%
    \put(0,0){\includegraphics[width=\unitlength,page=1]{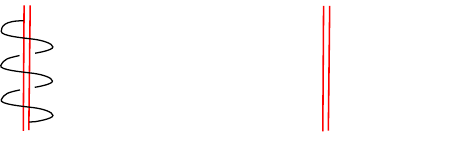}}%
    \put(0.07525749,0.27252957){\color[rgb]{1,0,0}\makebox(0,0)[lt]{\lineheight{1.25}\smash{\begin{tabular}[t]{l}$a+\varepsilon$\end{tabular}}}}%
    \put(0.07399657,0.22461529){\color[rgb]{1,0,0}\makebox(0,0)[lt]{\lineheight{1.25}\smash{\begin{tabular}[t]{l}$a$\end{tabular}}}}%
    \put(0.06895299,0.0052183){\color[rgb]{1,0,0}\makebox(0,0)[lt]{\lineheight{1.25}\smash{\begin{tabular}[t]{l}$a+\mu$\end{tabular}}}}%
    \put(0.71073934,0.27114043){\color[rgb]{1,0,0}\makebox(0,0)[lt]{\lineheight{1.25}\smash{\begin{tabular}[t]{l}$a+\varepsilon$\end{tabular}}}}%
    \put(0.12064997,0.13677243){\color[rgb]{0,0,0}\makebox(0,0)[lt]{\lineheight{1.25}\smash{\begin{tabular}[t]{l}$=(-1)^{(a+1)t}\varepsilon^{t+1}z^{\varepsilon t}\delta_{\varepsilon,\mu}$\end{tabular}}}}%
  \end{picture}%
\endgroup%

\end{center}
where  the arc on the left hand side has swift number $t$, and
$\delta_{\varepsilon,\mu}$ is the Kronecker symbol. \end{lemma}

 \begin{figure}[!h]
\centering
\def \svgwidth{.55\columnwidth}
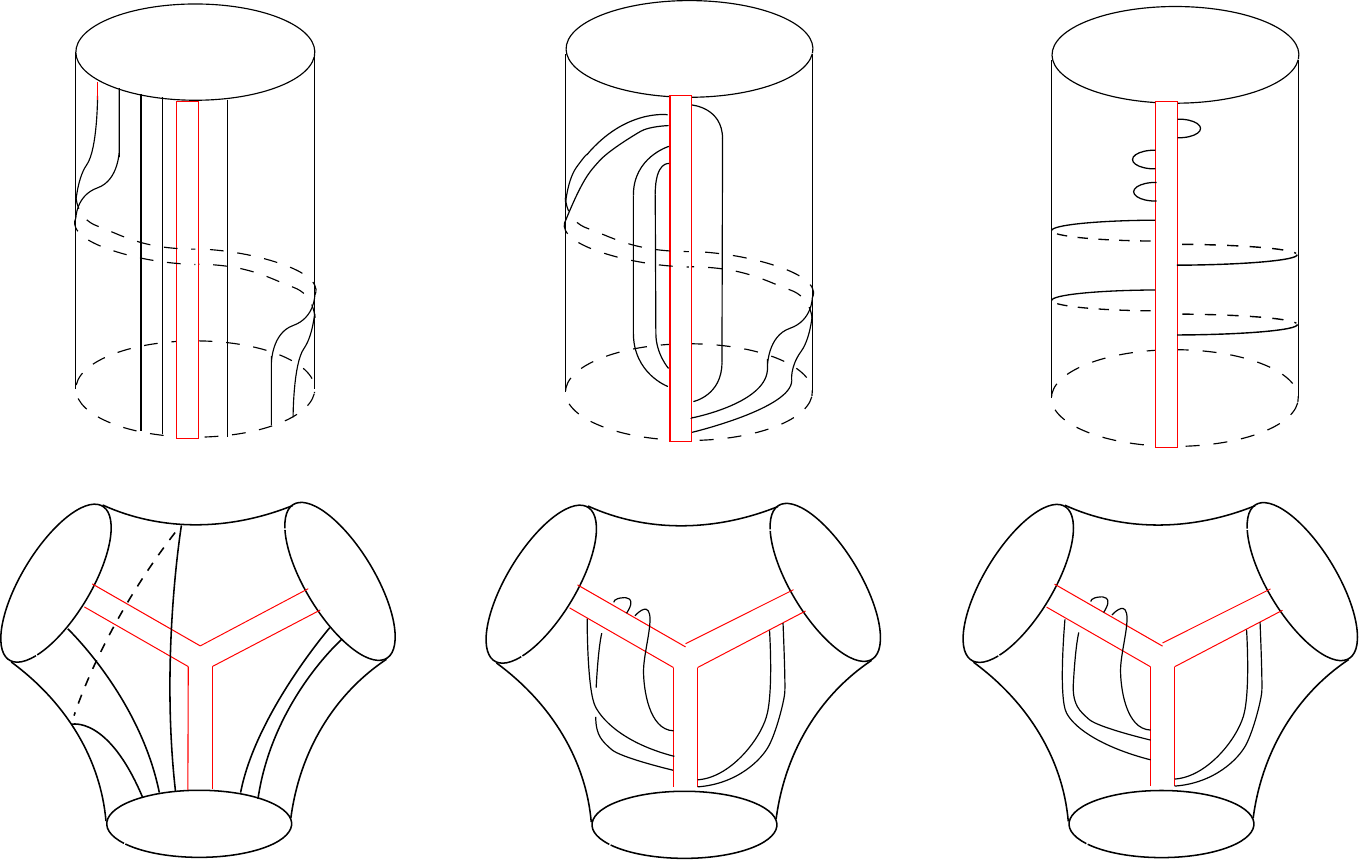
\caption{Patterns of $\gamma$ on  pieces of a decomposition system before (left) and after Step 1 (middle) and  after Step 2 (right).}
\label{fig:remainingpatterns}
\end{figure}

The strategy for computing the coefficients $C(\sigma)$ is as follows:
 \begin{itemize}
 \item [-] \underline{\emph{Step 1:}}
 For each $e\in E$, we start by applying the strand rule at each point of $\gamma \cap \alpha_e$ and $\gamma\cap \alpha_e'$ as well as at each come-back patterns of $\gamma$.
   The left   column of Figure  \ref{fig:remainingpatterns} shows patterns of $\gamma$ on an annulus and pants piece, while  the middle  column shows the resulting  configurations
  at the end of Step 1.
 \vskip 0.05in

  \item [-] \underline{\emph{Step 2:}} Apply the Siding Lemma to simplify the configurations resulting from Step 1.
  This lemma allows to slide and permute endpoints of remnants of $ \gamma$ on the edges of $\Gamma$ and to remove crossings between arcs of $\gamma$ resulting from Step 1.
 The right column of Figure  \ref{fig:remainingpatterns} shows the results of applying this step to the configurations resulting from Step 1.
 
 \vskip 0.05in

  \item [-] \underline{\emph{Step 3:}} 
  Use the half-twist, triangle and bigon  fusion rules, and Lemma \ref{lemma:singleArcAnnuli} 
  on the configurations resulting from Step 2, and collect terms corresponding to the same vectors
 $ \phi_{{\bf c}}$ of the fixed basis
  ${\mathcal B}_P$ to compute $C(\sigma)$.
  \end{itemize}
 
Next we   assign a \emph{weight} to each arc in $B$ that captures the ``cost"  of applying Steps 1-3 above.
Recall that

 \begin{equation}
 \label{triangles}
 \Delta_{++}=\frac{\left(\langle \frac{3\theta}{2}\rangle \langle \frac{\theta}{2}\rangle\right)^{\frac{1}{2}}}{\langle \theta \rangle},\  \Delta_{--}=-\frac{\left(\langle \frac{3\theta}{2}\rangle \langle \frac{\theta}{2}\rangle\right)^{\frac{1}{2}}}{\langle \theta \rangle} \ {\rm and} \  \Delta_{+-}=\Delta_{-+}=\frac{\langle \frac{\theta}{2}\rangle}{\langle \theta\rangle}.
 \end{equation}
 
Given $y\in B$ with endpoints $p_1,p_2$, and a state $\sigma\in \mathcal{S}_{\sigma}$, recall that the values  $s(p_1), s(p_2)$ are in $\{+1, -1\}$.
Slightly abusing notation, we will use $\Delta_{s(p_1),s(p_2)}$ to denote the quantity from Equation \eqref{triangles} for which the choice of the labels $\pm$ agrees with the signs of
$s(p_1),s(p_2)$. For instance, if $s(p_1)=+1$ and $s(p_2)=-1$, then $\Delta_{s(p_1),s(p_2)}=\Delta_{+-}$.
 
 \begin{definition}\label{def:weights} (\emph{Weights on arcs}) Given a state $s$ and an arc $y\in B$ with endpoints $p_1,p_2$,
 we define the \emph{weight} of $y$ with respect to $s$, denoted by  $w(y,s)$, as follows:
 \vskip 0.04in
 \begin{itemize}
\item[-]If $y$ lies on a pants piece of  $P\cup P'$, and  $p_1,p_2$ are on
$\gamma \cap P,$ then $$w(y,s):=s(p_1) s(p_2)\Delta_{s(p_1),s(p_2)}.$$ 
\item[-] If $y$ lies on  a  pants piece and $p_1$ is  on $P\cap \gamma$ while $p_2$ is at an undercrossing of a come-back pattern, then
$$w(y,s):=s(p_1) \Delta_{s(p_1),s(p_2)}.$$
\vskip 0.04in

\item[-] If $y$ lies on  a  pants piece and $p_1$ is  on $P\cap \gamma$ while $p_2$ is at an overcrossing of a come-back pattern, then
 $$w(y,s)= s(p_1) s(p_2)
 z^{2s(p_2)}\Delta_{s(p_1),s(p_2)}, \ \ {\rm where} \ \  z=e^{i\frac{\pi\theta}{2}}.$$
 \vskip 0.04in

\item[-]If $y$ lies on  an annulus piece, then 
$$w(y,s):=\begin{cases} 0,  \ \ \ \ \  \ \ \  \ \   \ \ \  \ \  \ \ \ \ \ \  {\rm if}\  s(p_1)\neq s(p_2), &\\
s(p_1)^{t+1}z^{ts(p_1)}, \ \ \  \   \   {\rm otherwise}.\end{cases}$$
 where $t=t_y$ is the swift number of the arc $y$.
 \end{itemize}
 \end{definition}

\begin{definition}\label{total weight} (\emph{Total weight of a state})With the notation and setting as above, we define the \emph{total weight} of a state $s\in \mathcal{S}_{\sigma}$ by
$W(s):=\underset{y\in B}{\prod}w(y,s).$
\end{definition}

We are now ready to give the state sum expressions for  $C(\sigma)$ promised in the beginning of the subsection.

\begin{proposition}\label{prop:statesum}Let $P\cup P'$, $\gamma$ and $B$  be fixed as above. Let $\phi_a$ be the TQFT vector corresponding to  $P$ with all edges colored by $a$, and let us consider the limit $r\rightarrow \infty, \frac{a}{r}\rightarrow \theta$ where $\theta$ is some fixed real number.
For  a total color shift $\sigma=(\sigma_e | e\in E)$,
let $C(\sigma)$ 
denote the
 limit of the coefficients of $T_r^{\gamma}\phi_a$ along $\phi_{a+\sigma}$.  Then,
$$ C(\sigma)=\varepsilon(\gamma)\underset{s\in \mathcal{S}_{\sigma}}  {\sum} W(s)=   \varepsilon(\gamma)     \underset{s\in \mathcal{S}_{\sigma}} {\sum}\ \  \ \underset{y\in B}{\prod}w(y,s).$$
where $\varepsilon(\gamma)\in \lbrace \pm 1 \rbrace$ depends only on the curve $\gamma.$
\end{proposition}

\begin{proof} Following \cite{RD}, we will view $T_r^{\gamma} \phi_a$ as $\gamma$ layered over  $\Gamma$ with each edge colored by $a$ and we will apply
 the three step process outlined earlier.  
 Recall that Step 1 is to apply
 the strand rule at each intersection point of $\gamma$ with the decomposition system $P\cup P'$, plus once for each come-back pattern, separating the overcrossing and undercrossing of each come-back pattern into two different arcs.
In Steps 2 and 3 we will apply  Lemmas \ref{lemma:sliding} and \ref{lemma:singleArcAnnuli} 
 and fusion rules to delete  black arcs in  the remnants of Step 1. Because  the Sliding Lemma allows to permute the endpoints of the configurations
 resulting from Step 1  along the edges of $\Gamma$, the order with which we  apply the strand rule at come-back arcs
 and for intersection points $\gamma\cap \alpha_e$ or $\gamma \cap \alpha_e'$ is not important. See Figure \ref{fig:remainingpatterns}. 
  
Let us first give an interpretation of states explaining, in particular, how conditions (a)-(b) of  Definition \ref{def:states}
arise in our setting. For each endpoint $p$ of an arc in  $B$ that is an intersection point of $ \gamma\cap (P\cup P')$ the sign $s(p)$ corresponds to the choice of 
$\pm 1$ color shift we used while applying the strand rule at $p$. If $p$ is an overcrossing (resp. undercrossing) of a come-back arc, $s(p)$ is the color shift (resp. $-1$ times the color shift) we used applying the strand rule at that come-back pattern. This produces a function $s: B_{\rm{end}}\longmapsto \lbrace \pm 1 \rbrace$ that  satisfies condition (b) of Definition \ref{def:states}.

Note that for each $s\in \mathcal{S}_{\sigma},$ the remaining pattern after the initial strand rules is a scalar multiple of the vector $\phi_{a+\sigma}$. 
To compute this scalar (up to sign), we have associated a weight to each arc, representing loosely the cost of removing that arc using fusion rules. Note that the strand rules themselves produce signs, which we need to distribute to the remaining arcs. Our convention is, that the sign of a strand rule corresponding to a come-back pattern, will be distributed to the overcrossing end, while the sign for a strand rule at a point $ \gamma\cap (P\cup P')$ will be distributed to the arc that is in a pants piece.

In Step 2, we use Lemma \ref{lemma:sliding} to \textit{separate} all remaining arcs. By that we mean that we want to put all arcs in annuli pieces in different slices of the annuli, and for arcs in pants pieces we want no crossing between two different black arcs (with respect to the projection induced by the dual graph). This process is illustrated in Figure \ref{fig:remainingpatterns}. 

Although the use of Lemma \ref{lemma:sliding} to separate arcs may create extra signs, we note that the total sign  depends only on the number of times we had to use the second part of Lemma \ref{lemma:sliding}, which depends only on $\gamma$ and not the state $s$.

Finally, we computed the coefficients needed to erase the remaining annuli arcs in Lemma \ref{lemma:singleArcAnnuli}; they match the weights of Proposition \ref{prop:statesum}. Now we explain how condition (a) of Definition \ref{def:states} is used in this process: By above discussion, after applying Steps 1-3  we have a function
$s: B_{\rm{end}}\longmapsto \lbrace \pm 1 \rbrace$. Note that if $s$  does not satisfy condition (a) of Definition \ref{def:states}, then there is an edge $e$ of $\Gamma$,
lying on an annulus piece of $P\cup P'$ with boundary consisting  of curves $\alpha_e, \alpha_e'$, such that
$$\underset{p\in \gamma \cap \alpha_e}{\sum} s(p)\neq \underset{p\in \gamma \cap \alpha_e'}{\sum} s(p).$$
Then,  when applying Lemma \ref{lemma:singleArcAnnuli}  to $e$, we will have instances where, and with the notation of the figure in the statement of Lemma \ref{lemma:singleArcAnnuli},
we will have $\epsilon\neq  \mu$
yielding  zero coefficients and zero contribution to $C(\sigma)$. Hence only functions $s$ that satisfy condition (a) of Definition \ref{def:states} contribute to $C(\sigma)$.

To erase an arc in a pants piece with two endpoints in $\gamma\cap(P\cup P'),$ a single triangle rule is needed, which, together with the strand rule signs, gives exactly the weight $w(y,s)$ of Definition \ref{def:weights}. 
Finally, to erase an arc from an overcrossing or undercrossing to an intersection point $\in \gamma\cap(P\cup P')$, we need to apply a half-twist rule, then a triangle rule. The overcrossing and undercrossing of a come-back arc yield exactly the same half-twist coefficient. It will be convenient for us to move both coefficients into the weight of the arc of the overcrossing. This gives the second and third rule of Definition \ref{def:weights}, remembering that we also assign the strand rule sign to overcrossings. 
\end{proof}

\subsection{A Laurent polynomial}The second main result of the section is the following proposition the proof of which  uses
Proposition \ref{prop:statesum}.

\begin{proposition}\label{prop:polynomial} Let the notation and setting be as in Proposition \ref{prop:statesum}. For any
total color shift $\sigma$, there exists  $n,m\in \ZZ$ and a Laurent   polynomial $P_{\sigma}(z)\in {\mathbb Q}[z^{\pm1}]$
such that
$$C(\delta)=\Delta_{++}^n\, \Delta_{+-}^m\,  P_{\sigma}(z),$$
at $z=e^{i\frac{\pi\theta}{2}}.$
Here, $\Delta_{++},\Delta_{+-}$ are as defined in Equation \eqref{triangles}. 
\end{proposition}

\begin{proof} 
We will deduce this proposition from Proposition \ref{prop:statesum}.
Note that by Definition \ref{def:weights}, for any  $s \in \mathcal{S}_{\sigma},$ the total weight is of the form $$W(s)=\pm \Delta_{++}^{n_1(s)}\, \Delta_{+-}^{n_2(s)}\, \Delta_{--}^{n_3(s)}z^{m(s)},$$

where $\Delta_{++},\Delta_{+-}$ and  $\Delta_{--}$ are as defined in Equation \eqref{triangles}.
Moreover, we claim that $n_1(s)+n_2(s)+n_3(s)=|B_{\textrm{pants}}|,$ where $B_{\textrm{pants}}$ is the set of arcs of $B$ that belong to a pants piece of the decomposition system. Indeed, each triangle coefficient $\Delta_{\varepsilon \mu}$ is created when erasing an arc of $B_{\textrm{pants}}.$

Next we claim that 

\begin{equation}\label{tnumber}
n_2(s)+2n_3(s)=K_{\gamma}+\underset{e \in E}{\sum} \left(I(\alpha_e,\gamma)-\sigma_e\right),
\end{equation}
where $K_{\gamma}$ is the number of come-back patterns of $\gamma$. Hence, in particular $n_2(s)+2n_3(s)$ is independent of the state
 $s\in \mathcal{S}_{\sigma}$.

 To prove the claim, first note that by Definition \ref{def:states},  the triangle rules of Figure \ref{fig:fusionRules2} and by Equation \eqref{triangles}, it follows that
 the number $n_2(s)+2n_3(s)$  is the same as the number of endpoints $p$ of arcs in $B_{\textrm{pants}}$ such that $s(p)=-1$.   
Thus it is enough to argue that this later number is equal to the quantity in the right hand side of Equation \eqref{tnumber}.

Recall that pants pieces contain two types of arcs; come-back patterns and arcs that have their endpoints on two different boundary components of a pair of pants. Consider any state $s\in \mathcal{S}_{\sigma}$.  Because $s$  satisfies condition (b) of Definition \ref{def:states},
exactly one of the two endpoints of any come-back arc has a $-$ sign.
Moreover, for any edge $e\in E$, with total shift $\sigma_e$  at $e$, we have
$$\sigma_e=\underset{p \in \alpha_e \cap \gamma}{\sum}s(p).$$ 
The total number of terms  in the right hand side of above equation,
is the  intersection number $I(\alpha_e,\gamma)$. Hence,
$s$ takes the value $+1$ exactly $\frac{1}{2}\left(I(\alpha_e,\gamma)+\sigma_e\right)$ times on $\alpha_e\cap \gamma$, and the value $-1$ exactly $\frac{1}{2}\left(I(\alpha_e,\gamma)-\sigma_e\right)$ times on $\alpha_e \cap \gamma$.  Finally,  $s$ also takes the value $-1$ exactly $\frac{1}{2}\left(I(\alpha_e,\gamma)-\sigma_e\right)$ times on points of $\alpha_e'\cap \gamma$ by the same reasoning and by condition (a) of Definition \ref{def:states}.

Let us  now pick a state $s_0\in \mathcal{S}_{\sigma}$ such that $n_2(s_0)$ is maximal, and let us set $N_i:=n_i(s_0)$, for $i=1,2, 3$. From the above discussion and some elementary linear algebra, we deduce that one can write for any $s\in\mathcal{S}_{\sigma}:$
$$n_1(s),n_2(s),n_3(s)=N_1+k(s),N_2-2k(s),N_3+k(s)$$
where $k(s)$ is a non-negative integer. Thus for any $s\in\mathcal{S}_{\sigma},$ we have that $\frac{W(s)}{\Delta_{++}^{N_1+N_3}\Delta_{+-}^{N_2}}$ is of the form $\pm z^l \left(\frac{\Delta_{++}^2}{\Delta_{+-}^2}\right)^k$ where $l\in \ZZ$ and $k\geq 0.$ 
We conclude noting that
$$\frac{\Delta_{++}^2}{\Delta_{+-}^2}=\frac{\langle \frac{3\theta}{2}\rangle}{\langle \frac{\theta}{2}\rangle}=\frac{e^{\frac{3i\pi\theta}{2}}-e^{-\frac{3i\pi\theta}{2}}}{e^{\frac{i\pi\theta}{2}}-e^{-\frac{i\pi\theta}{2}}}=e^{i\pi\theta}+1+e^{-i\pi\theta}=z^2+1+z^{-2}.$$
\end{proof}


\section{Producing non vanishing coefficients} \label{sec:five}
In this section we apply the settings and results of
 Section \ref{sec:fusion_computations} 
to prove Theorem \ref{thm:lower}.

\subsection{Proof outline}
Fix a decomposition system $P\cup P'=\{\alpha_e, \alpha'_e \  | \ e\in E\}$ of $\Sigma$, and
suppose that
$\gamma$ is a multicurve in Dehn-Thurston position with respect to it.
As before we have   an arc system $B=B(\gamma)$ for $\gamma$. For any color shift $\sigma$ on $E$ we can define the set of states
${\mathcal S}_{\sigma}$ on $B$ and, as in Proposition \ref{prop:statesum}, the limits $C(\sigma)$ of  the coefficients  $T_r^{\gamma}$.

Let $M(\gamma, P)$ be as in Equation \eqref{eq:maxint} and let $e_0\in E$ such that $M(\gamma, P)=I(\gamma, \alpha_{e_0})$.

For any $\abs{\delta}\leq M((\gamma, P) $ consider the total color shift
$\sigma_{\delta}=(\sigma_e| e\in E)$ given by

\begin{equation}
\label{Qstates}
\sigma_{\delta}=\begin{cases} \delta \ \ \ \ \  \ \ \  \ \  {\rm if}\ \  e=e_0,&\\
I(\gamma, \alpha_{e})\ \ {\rm if} \ \  e\neq e_0.\end{cases},\end{equation}
where $\alpha_e \in P$ is the curve dual to the edge $e\in \Gamma$.
Since, with the items fixed earlier, $\sigma_{\delta}$ varies only as $\delta$ does,
we will use   the notation $C(\delta):=C(\sigma_{\delta})$ and ${\mathcal S}_{\delta}:={\mathcal S}_{\sigma_{\delta}}$.

Theorem \ref{thm:lower} follows immediately from the following.
\begin{theorem}\label{thm:lowerlimit} 
For each   $\abs{\delta}\leq M(\gamma, P)$ with $\delta \equiv M(\gamma, P)(\rm{mod}\ 2)$,  we have $C(\delta)\neq   0$.
\end{theorem}

We begin by noting that the fact that $M(\gamma, P)=I(\gamma, \alpha_{e_0})$ puts restrictions on the types of arcs of $\gamma$
allowed on the pieces of $P\cup P'$ bordering $ \alpha_{e_0}$ or the parallel copy  $\alpha'_{e_0}$.

\begin{lemma}\label{lem:come-backArcs} Suppose that  $\gamma$ is in Dehn-Thurston position with respect to $P\cup P'$ and that $M(\gamma, P)=I(\gamma, \alpha_{e_0})$, for some
$e_0\in E$. Then, we have the following:
\begin{itemize}
\item [(a)] Suppose $ \alpha_{e_0}$ and $\alpha'_{e_0}$ belong on two different pants pieces of the decomposition. Then,  any
 come-back pattern of $\gamma$ on the pants containing $\alpha_{e_0}$ (or $\alpha'_{e_0}$ ) must have their
 endpoints on $\alpha_{e_0}$ (or $\alpha'_{e_0}$ ).
 \item [(b)] Suppose $ \alpha_{e_0}$ and $\alpha'_{e_0}$ belong on a single pants piece of the decomposition. Then there is no come-back pattern of $\gamma$ on this pants piece.
\end{itemize}
\end{lemma} 
\begin{proof}

We observe that if a pants piece of $P\cup P'$ has boundary curves $\alpha_e,\alpha_f,\alpha_g,$ then there is one (or more) come-back arcs with two endpoints on $\alpha_e$
 if and only if $$I(\alpha_e,\gamma)>I(\alpha_f,\gamma)+I(\alpha_g,\gamma).$$ This inequality  is an immediate consequence of the fact that $\gamma$ is into Dehn-Thurston position and it  implies the first part (a) of the lemma.

Suppose now that  $\alpha_{e_0}$ and $\alpha'_{e_0}$ are boundary curves on a single pants piece and let  $\alpha_f$ denote the third boundary curve of the pants.
The hypothesis  $M(\gamma, P)=I(\gamma, \alpha_{e_0})$,  implies  that there are no come-back arc in that pants with two endpoints on $\alpha_{f}.$
Since, by the discussion above,  $$I(\alpha_{e_0},\gamma)\leq I(\alpha_{e_0},\gamma) +I(\alpha_f,\gamma),$$ there is also no come-back arc in that pants with two endpoints on $\alpha_{e_0}$ or  $\alpha'_{e_0}$.
\end{proof}

\vskip 0.05in

Given $C(\delta)$ as in the statement of Theorem \ref{thm:lowerlimit}, by Proposition \ref{prop:polynomial} there is a Laurent polynomial $P_{\delta}(z)\in {\mathbb Q}[z^{\pm1}]$
so that
$$C(\delta)=\Delta_{++}^n\, \Delta_{+-}^m\,  P_{\delta}(e^{i\frac{\pi\theta}{2}}),$$
for some $m,n\in \ZZ$ and $z=e^{i\frac{\pi\theta}{2}}.$
We will show the following:

 \begin{proposition}\label{prop:nozeropoly}For each   $\abs{\delta}\leq M(\gamma, P)$ with $\delta \equiv M(\gamma, P)(\rm{mod}\ 2)$, 
 we have $P_{\delta}(z)\neq 0$ in ${\mathbb Q}[z^{\pm1}]$.
 \end{proposition}

Having Proposition \ref{prop:nozeropoly} at hand, Theorem \ref{thm:lowerlimit} is derived as follows: If $C(\delta)$ is zero then $P_{\delta}(e^{i\frac{\pi\theta}{2}})=0$.
Since we can choose $\theta$ so that $e^{i\frac{\pi\theta}{2}}$ is a transcendental number we arrive at contradiction.

The remaining of the section will be devoted to the proof of Proposition \ref{prop:nozeropoly}.

\subsection{Two lemmas and set up} We will use $\deg_z(P_{\delta}(z))$ to denote the highest power of $z$ in $P_{\delta}(z)$.
By Proposition \ref{prop:statesum} we have

$$ C(\delta)=\varepsilon(\gamma)\underset{s\in \mathcal{S}_{\delta}}  {\sum} W(s),$$
where
$W(s)= \underset{y\in B}{\prod}w(y,s),$ $\varepsilon(\gamma)\in \lbrace \pm 1 \rbrace$
and the arc weights are given in Definition 
\ref{def:weights}.

Note that for any state
the  total weight $W(s)$ is of the form
$$W(s)=\pm \Delta_{++}^{n_1(s)}\, \Delta_{+-}^{n_2(s)}\,  z^{m_s},$$
where $m_s\in \ZZ$. The proof of Proposition \ref{prop:polynomial} shows that after multiplying all weights by a common non-zero factor $\Delta_{++}^N\, \Delta_{+-}^M,$ we get monomials $P(s)$ with $\deg_z(P(s))=m_s+n_1(s)$ up to a constant that does not depend on $s.$ Therefore we will set 
$$\deg_z(W(s)):=m_s+n_1(s)$$
or equivalently
\begin{equation}\label{eq:degTriangCoeff}\deg_z(\Delta_{\varepsilon, \mu})=\delta_{\varepsilon,\mu}=\frac{1+\varepsilon\mu}{2}
\end{equation}
where $\delta$ is the Kronecker symbol and the second equality holds since we consider $\varepsilon,\mu\in \lbrace \pm 1 \rbrace.$ Also, considering the sign of the coefficient of highest degree in $z$ after renormalization, we will set:
\begin{equation}\label{eq:signTriangCoeff}\sign(\Delta_{++})=\sign(\Delta_{+-})=+1, \ \textrm{and} \ \sign(\Delta_{--})=-1.
\end{equation}

To prove Proposition \ref{prop:nozeropoly} for each $\delta\leq \abs{M(\gamma, P)}$, we will determine all the states in ${\mathcal S}_{\delta}$ that contribute to $\deg_z(P_{\delta}(z))$ 
and show that there is no cancellation between the total weights corresponding to these states. From the above, we can look for states with $\deg_z(W(s))$ maximal, and show that the corresponding terms do not cancel out.

We begin with two lemmas that describe some  properties of  the states in ${\mathcal S}_{\delta}$.

\begin{lemma}\label{lemma:plusatover}Suppose that $\abs{\gamma \cap \alpha_{e_0}}>1$. Let $s, s'$ be two states in ${\mathcal S}_{\delta}$ such that
there are $p_2,q_2\in B_{\rm{end}}$ corresponding to the overcrossing and undercrossing of the same come-back arc of $\gamma$, respectively,
so that $s(p_2)=+1$ and $s(q_2)=-1$ while $s'(p_2)=-1$ and $s'(q_2)=+1$. Suppose, moreover, that for any $q\ne p_2, q_2$, we have $s(q)=s'(q)$. 
 Then, we have  $\deg_z(W(s))< \deg_z(W(s'))$.
\end{lemma} 
\begin{proof}
Suppose that $p_2$ and $q_2$ belong  to arcs $y$ and $y'$ in $ B$, respectively. Let $p_1, q_1 \in B_{\rm {end}}$ denote the  second endpoints of $y$ and $y',$ respectively.

Suppose that $s(p_1)=s'(p_1)=\mu$ and  $s(q_1)=s'(q_1)=\epsilon$, where $\epsilon, \mu \in \{\pm 1\}$.

By assumption,
 the only difference between $W(s)$ and $W(s')$ is the product of weights on the arcs $y, y'$, where
we have $$w(y,s)\,  w(y',s)=\epsilon \mu z^{2} \, \Delta_{+\mu} \Delta_{-\epsilon}, \ \ {\rm while} \ \
w(y,s')\,  w(y',s')=-\epsilon \mu z^{-2}\, \Delta_{-\mu} \Delta_{+\epsilon}.$$
By Equation \eqref{eq:degTriangCoeff}, we get:
$$\deg_z(w(y,s)\,  w(y',s))=3+\frac{\mu-\varepsilon}{2}, \ \ {\rm while} \ \
\deg_z(w(y,s')\,  w(y',s'))=-1-\left(\frac{\mu-\varepsilon}{2}\right)$$
Since $\mu,\varepsilon\in \lbrace \pm 1 \rbrace,$ we get:
$$\deg_z\big(w(y,s')\,  w(y',s')\big) \leq \deg_z\big(w(y,s)\, w(y',s)\big) -2.$$
and the lemma follows.
\end{proof}

By Lemma \ref{lemma:plusatover}, while looking to identify the states that contribute to the maximum degree of $P_{\delta}(z)$,
it is enough to look at states that assign the value $+1$ to all points in $B_{\rm{end}}$ that correspond to overcrossings of some come-back pattern of $\gamma$.

With this in mind, we define
$ {\mathcal S}^{+}_{\delta}$ to be the set of all 
$s\in {\mathcal S}_{\delta}$ such that we have $s(p)=+1$ for any point $p\in B_{\rm{end}}$ that is an overcrossing of a come-back pattern of $\gamma$. 
Hence, by Part (2) of Definition \ref{def:states},
 for any $q$ that is an undercrossing of a come-back pattern of $\gamma$ we have $s(q)=-1$.

\begin{lemma}\label{properties}We have the following:
\begin{enumerate}
\item[(a)] Let $s\in{\mathcal S}_{\delta}$. For  any edge $e\neq e_0$ of $\Gamma$ and any $p\in (\gamma\cap \alpha_e)$,
we have $s(p)=1$. 

\item[(b)] For any two states $s_1, s_2\in{\mathcal S}^{+}_{\delta}$ and any arc $y\in B$ that has none of its endpoints on  $\alpha_{e_0}\cup \alpha'_{e_0}$,
we have $w(y, s_1)=w(y, s_2)$.   
\end{enumerate}
\end{lemma}
\begin{proof}Part (a) follows from Equation \eqref{sumcurve} in Definition \ref{def:states}, and the fact that, by the Equation \eqref{Qstates}  we have
$\sigma_e=I(\gamma, \alpha_{e})$. Part (b) follows from (a) and Definition \ref{def:weights}.
\end{proof}

\begin{figure}[h]
\centering
\def \svgwidth{.8\columnwidth}
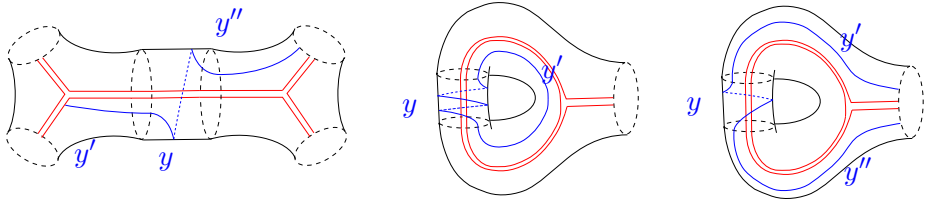
\caption{Arcs of $\gamma$ and the edge $e_0$ in $\Sigma$.}
\label{fig:ArcCases}
\end{figure}

To continue let $B_{e_0}$ denote the subset of the arcs system  $B$ consisting of arcs that have at least one endpoint on $\alpha_{e_0}\cup\alpha'_{e_0}$. 
An arc $y\in B_{e_0}$  can have both endpoints on $\alpha_{e_0}\cup\alpha'_{e_0}$ and it lies on the annulus  piece of the decomposition system that is dual to the edge $e_{0}$,
or can have only one endpoint on $\alpha_{e_0}\cup\alpha'_{e_0}$ and it lies on a pants piece adjacent to the annulus.
The arcs $y\in B_{e_0}$ with both endpoints on $\alpha_{e_0}\cup\alpha'_{e_0}$ are exactly the arcs for which we have $y\cap \alpha_{e_0}=y\cap \alpha'_{e_0}\neq \emptyset$.

Next we will summarize  the properties of these arcs that are useful to us. There are two cases to consider according the whether
the edge $e_0$ connects two vertices of $\Gamma$ or not. See Figure \ref{fig:ArcCases}.
\begin{itemize}

\item [(I)]The edge $e_0$ is a loop in $\Gamma$ and lies on an one-holed torus of $\Sigma$. The intersection of $\gamma$ with the one-holed torus component  consists of either simple closed curves included in the one-holed torus, or arcs that start and end on the boundary component $\alpha_f$ of the one-holed torus. In the first case, we may decompose the curve into arcs $y_1,z_1,\ldots,y_p,z_p,$ where $y_i$ lies inside the pants delimited by $\alpha_{e_0},\alpha_{e_0}'$ and $\alpha_f,$ while $z_i$ lies in the annulus piece bounded by $\alpha_{e_0}$ and $\alpha_{e_0}'.$ In the second case we decompose the arc into arcs $y_0,z_1,y_1,\ldots y_{p-1},z_p,y_p$ similarly defined except $y_0$ and $y_p$ have an extremity on $\alpha_f.$

\item[(II)] The edge $e_{0}$ connects two distinct vertices of $\Gamma$ corresponding to two distinct pants pieces of the decomposition system.
The endpoints of an arc $y$ that lies on the annulus piece dual to $e_{0}$, are connected to two arcs $y', y'' \in B_{e_0} $ that lie on distinct pants pieces of the decomposition. In this case, we define $u_y$ to be the number of
endpoints of $y,y''$ that are undercrossings of a come-back pattern of $\gamma$ and we define $o_y$ to be the number of
endpoints of $y,y''$ that are overcrossings of a come-back pattern of $\gamma$. We also define $v_y:=0$.
\end{itemize}

In the case where $e_0$ is not an $1$-edge loop, there is a concise description of states of maximum degree, that can be defined using the following notion of complexity:
\begin{definition}\label{def:complexity}Assume that $e_0$ is not an 1-edge loop. Given  $y\in B_{e_0}$  that lies on the annulus piece of the decomposition dual to $e_{0}$, define the complexity
$$\tau(y\cap \alpha_{e_0}):=t_y-u_y+v_y$$
where $t_y$ is the swift number of $y$ (see Definition \ref{def:swift}), and $u_y$, $v_ y$ are as defined above. 
\end{definition}

We will partition the set of points  in $\gamma\cap \alpha_{e_0}\neq \emptyset$ into disjoint sets
$T_1, T_2, \ldots, T_k$ such that, for $ 1\leq i\leq k-1$,
\begin{itemize}
\item [(i)]  all the points  $T_i$ have the same complexity $\tau_i$; and
\item[(ii)] the points in $T_i$  have complexity strictly smaller that the points in $T_j$, for any $i<j\leq k-1$. That is $\tau_i<\tau_j$.
\end{itemize}

\subsection{Determining the leading term} Next we  we will determine all the states in ${\mathcal S}_{\delta}$ that contribute to the maximum degree $\deg_z(P_{\delta}(z))$.
By definition for any $s\in {\mathcal S}_{\delta}$,  the only negative values of $s$ on points in $\gamma\cap P$ occur on points in $\gamma\cap \alpha_{e_0}$
and the number of these negative values is $l=\frac{1}{2}(I(\gamma, \alpha_{e_0})-\delta)$.

\begin{definition}\label{def:maxstates}
Let ${\mathcal S}_{{\rm max}}\subset {\mathcal S}^{+}_{\delta}$ denote the set of all states such that for any $p, q\in \gamma\cap \alpha_{e_0} $, we have: If $s(p)=-1$ and $s(q)=1$ then $\tau(p)<\tau(q)$.
\end{definition}
\begin{remark} \rm{We can construct the states of Definition \ref{def:maxstates} as follows:  Begin by assigning values -1 to points in the set $T_1$ of minimum complexity.
If  $l\leq |T_1|$ then we have $|T_1|\choose l$ states in ${\mathcal S}_{{\rm max}}$. In general, if 
$l$ is less or equal to the cardinality, say $\kappa$, of the set $\cup_{1=1}^j T_i$ and $l-\kappa\leq |T_{j+1}|,$ then we have $|T_{j+1}|\choose {l-\kappa}$ states in ${\mathcal S}_{{\rm max}}$.}
\end{remark}

We have the following lemma which implies Proposition \ref{prop:nozeropoly} and hence Theorem \ref{thm:lowerlimit}, in the case where $e_0$ is not a $1$-edge loop.

\begin{lemma}\label{lem:maxdegree} 
Assume that $e_0$ is not a $1$-edge loop and $\abs{\gamma\cap \alpha_{e_0}}\geq 1$.
Then, we have 
$$  \underset{s\in \mathcal{S}{{}}}  {\sum} W(s) \neq 0.$$
Furthermore, for  any $s\in  \mathcal{S}_{{\rm max}}$ and  any $s'\notin\mathcal{S}_{{\rm max}}$,
we have $ \deg_z(W(s'))< \deg_z(W(s))$.
\end{lemma}

\begin{proof}To prove the lemma it is enough
 to prove the statements (a)-(b) below.
\begin{enumerate}
\item[(a)]  For any state $s\in  \mathcal{S}_{{\rm max}}$ and  any $s' \notin \mathcal{S}_{{\rm max}}$
we have $ \deg_z(W(s'))< \deg_z(W(s))$.
\item [(b)] The terms in the sum $\underset{s\in \mathcal{S}_{{\rm max}}}  {\sum} W(s)$
have the same sign; so there is no cancellation.
\end{enumerate}
Next we claim that item (a) will follow from Definition \ref{def:maxstates} and the following statement:
\begin{itemize}
\item[(c)]There is a constant $B_{\gamma}$ depending only on $\gamma$ such that for all $s\in \mathcal{S}_{\delta}^+,$ we have 
$$\deg_z(W(s))=B_{\gamma}+\underset{p\in \gamma\cap \alpha_{e_0}}{\sum}s(p)\tau(p).
$$
\end{itemize}
Indeed, lets us first note that by Lemma \ref{lemma:plusatover}, it is enough to consider states in $\mathcal{S}_{\delta}^+,$ since for $s\notin \mathcal{S}_{\delta}^+,$ there is $s'\in \mathcal{S}_{\delta}^+$ such that $\deg_z(W(s))<\deg_z(W(s'))$. Namely, we can take $s' \in \mathcal{S}_{\delta}^+$ that differs from $s$ only at overcrossings and undercrossings. 

Now, assuming that (c) is true, then for a state $s$ to  maximize $\deg_z(W(s))$ one needs to maximize the quantity $$\underset{p\in \gamma\cap \alpha_{e_0}}{\sum}s(p)\tau(p).$$ 
But for any $p\in \gamma \cap \alpha_{e_0}$, we have $s(p)\in\lbrace \pm 1 \rbrace$. Furthermore,  for any $s\in \mathcal{S}^+_{\delta}$,
there are the same number of $p\in \gamma \cap \alpha_{e_0}$ such that $s(p)=+1$ .
Hence to maximize the above  sum, one needs to distribute the $p$'s for which $s(p)=+1$ in decreasing order of the $\tau(p)$, which is, by Definition \ref{def:maxstates}, is  the same as requesting that $s\in \mathcal{S}_{{\rm max}}$. This finishes the proof of the claim that statement (c) implies statement (a).
\vskip 0.03in

Let us now turn to proving (b) and (c):
For any $s\in {\mathcal S}^{+}_{\delta}$
set
$$W_{e_0}(s):= \underset{y\in B_{e_0}}{\prod}w(y,s).$$
By Lemma \ref{properties} and for every state $s\in \mathcal{S}^+_{\delta},$ we can write $W(s)$ as a product of $W_{e_0}(s)$ and a factor that doesn't depend on $s$.
Thus  to finish the proof of the lemma it is enough to prove item (b) and (c) using $W_{e_0}(s)$ instead of $W(s).$

Note that the type of arcs in the pants with $\alpha_{e_0}$ have the form described in the first case of Figure \ref{fig:ArcCases} thanks to Lemma \ref{lem:come-backArcs}.

Then we can partition the arcs in $B_{e_0}$ into triples $ y, y' , y''$ as described in case (II) of Definition \ref{def:complexity} and we write
$$W_{e_0}(s):= \underset{y, y', y''\in B_{e_0}}{\prod}\big( w(y,s)\, w(y',s)\, w(y'',s)\big).$$
Given such a triple of arcs  and $s\in {\mathcal S}^{+}_{\delta}$, we
let
 $\epsilon_y:=s(y\cap \alpha_{e_0})=s(y\cap \alpha'_{e_0})$, which is the value of $s$ for both end points of $y$ and for one endpoint of each of $y',y''$.
 Let $\mu$ and $\nu$ denote the values that $s$ assigns to the second endpoint of $y'$ and $y''$ respectively.

By Definition \ref{def:weights}, and since $s$ assigns +1 to all overcrossing points, we have
 $$w(y,s)\, w(y',s)\, w(y'',s)=\  \big(\epsilon_y\big)^{t_y+1}\, z^{t_y  \epsilon_y} \, z^{2o_y} \, \Delta_{\mu, \epsilon_y}\ ,\Delta_{\epsilon_y,  \nu}.
 $$
By Equation \eqref{eq:degTriangCoeff}, we get:
\begin{eqnarray}\label{eq:degCase2}
\nonumber \deg_z(w(y,s)\, w(y',s)\, w(y'',s))&=&t_y\epsilon_y+2o_y+\frac{1+\epsilon_y\mu }{2}+\frac{1+\epsilon_y \nu}{2}
\\ \deg_z(w(y,s)\, w(y',s)\, w(y'',s))&=&t_y\epsilon_y+2o_y+1+\epsilon_y(1-u_y)
\end{eqnarray}
where to get the second equality we used that $\mu,\nu=-1$ if the second endpoint of $y',y''$ is an undercrossing, and $\mu,\nu=+1$ else.
This implies that $\frac{\mu+\nu}{2}=1-u_y$, for any $\mu, \nu \in\{\pm1\}$ and $u_y\in \{0,1,2,\}$.
Moreover, by Equation \eqref{eq:signTriangCoeff}, we get
\begin{equation}\label{eq:signCase2}
\sign(w(y,s)\, w(y',s)\, w(y'',s))=\epsilon_y^{t_y+1+u_y}=\epsilon_y^{t_y+1-u_y}
\end{equation}
Combining the contribution of all arcs in $B_{e_0}$ we get:

\begin{eqnarray*}
\deg_z(W_{e_0}(s))&=&\underset{y\in B_{e_0}}{\sum} (t_y  \epsilon_y-\epsilon_yu_y+ 2o_y+ 1+\epsilon_y)
\\ &=&2O_{e_0}+M(\gamma,P)+\delta+\underset{y\in B_{e_0}}{\sum} \epsilon_y \tau_y
\end{eqnarray*}
where $O_{e_0}$ is the total number of overcrossing points in $ B_{e_0}$, $M(\gamma,P)$ is given by  Equation 
\eqref{eq:maxint}, and we have set
$\tau_y:=\tau(y\cap \alpha_{e_0})$. This proves item (c) and thus only the states in $ \mathcal{S}_{{\rm max}}$ contribute to the maximum degree
 of  $P_{\delta}(z)$.
 
 Finally, by Equation \eqref{eq:signCase2} we have
  
\begin{eqnarray*}
{\rm sign}(W_{e_0}(s))&=& \underset{y\in B_{e_0}} {\prod}\ \big(\epsilon_y\big)^{t_y+1-u_y}
\\ &=&\underset{y,  \epsilon_y=-1}{\prod} (-1)^{\tau_y+1}.
\end{eqnarray*}
 This quantity is the same  for all $s\in  \mathcal{S}_{{\rm max}}$,
by the definitions, finishing the proof of (b).
\end{proof}

\subsection{Loop-edge}
The case where $e_0$ is a one-edge loop is similar, except we cannot give an explicit description of maximal states. However, it is still possible to show there is no cancellation among maximal states and prove:
\begin{lemma}\label{lem:maxdegreeCase2} 
Assume that $e_0$ is a $1$-edge loop and $\abs{\gamma\cap \alpha_{e_0}}\geq 1$.
Then we have 
$$  \underset{s\in \mathcal{S}}  {\sum} W(s) \neq 0.$$
\end{lemma}
Note that thanks to Lemma \ref{lem:come-backArcs}, the arcs in the pants containing $\alpha_{e_0}$ are of the form described in the second and third drawing of Figure \ref{fig:ArcCases}. We should remark that although it is not represented in the picture, they can have arbitrarily large intersection with $\alpha_{e_0}.$
\begin{proof}
We will prove that states of maximal degree all appear with the same sign. Since maximal states are all in $\mathcal{S}_{\delta},$ it is sufficient to compute the degree and sign of terms $W(s)$ with $s$ in $\mathcal{S}_{\delta}.$ Let us consider the collections $Y_1$ and $Y_2$ of components of the intersection $\gamma$ with the one-holed torus component that is the connected component of $\alpha_{e_0}$ in $\Sigma\setminus (P\setminus \alpha_{e_0}),$ where $Y_1$ corresponds to simple closed curves and $Y_2$ to arcs that start and end on $\alpha_f$ the boundary component of the one-holed torus.

Let us consider a closed curve $y$ in $Y_1,$ decomposed as $y_1 z_1 \ldots y_p z_p$ where $y_i$ are in the pants piece and $z_i$ are in the annulus piece. Let us call $\varepsilon_i$ the values of the state on the intersection points in $z_i\cap (\alpha_{e_0} \cup \alpha_{e_0}'),$ remembering that these values are equal at start and end of any arc in an annulus piece. We have
$$w(y,s):=\underset{i=1}{\overset{p}{\prod}}w(y_i,s)\, w(z_i,s)=\underset{i=1}{\overset{p}{\prod}}(\varepsilon_i)^{t_i+1}\, z^{\varepsilon_i t_i}\,  \Delta_{\varepsilon_{i-1},\varepsilon_{i}}$$
where $\varepsilon_{0}=\varepsilon_p$ by definition and $t_i$ is the swift number of the annulus arc $z_i,$ therefore
$$\deg(w(y,s))=\underset{i=1}{\overset{p}{\sum}} \left(t_i\varepsilon_i +\frac{1+\varepsilon_{i-1}\varepsilon_i}{2}\right)$$
and 
$$\sign(w(y,s))=(-1)^{\underset{i=1}{\overset{p}{\sum}} \left(\frac{(t_i+1)(\varepsilon_i-1)}{2}+\frac{(\varepsilon_i-1)(\varepsilon_{i-1}-1)}{4}\right)}=(-1)^{\frac{1}{2}\underset{i=1}{\overset{p}{\sum}} \left(t_i\varepsilon_i +\frac{1+\varepsilon_{i-1}\varepsilon_i}{2}\right)-\frac{p}{2}-\frac{1}{2}\underset{i=1}{\overset{p}{\sum}} t_i},$$
where we used that for $\varepsilon\in \lbrace \pm 1\rbrace$ we have $\varepsilon=(-1)^{\frac{\varepsilon-1}{2}},$ and moreover $\sign(\Delta_{\varepsilon,\mu})=(-1)^{\frac{(\varepsilon-1)(\mu-1)}{4}}.$
Similarly, for an arc $y$ in $Y_2$ decomposed as $y_0z_1\ldots z_p y_p$ we have 
$$w(y,s):=w(y_0,s)\, \underset{i=1}{\overset{p}{\prod}}w(y_i,s)\, w(z_i,s)=\Delta_{\varepsilon_0,\varepsilon_1}\underset{i=1}{\overset{p}{\prod}}(\varepsilon_i)^{t_i+1}\, z^{\varepsilon_i t_i} \Delta_{\varepsilon_{i},\varepsilon_{i+1}}$$
where by definition $\varepsilon_0=\varepsilon_p=+1$ since those shifts correspond to intersection points on $\alpha_f$ and we also get that $$\sign(w(y,s))=(-1)^{deg(w(y,s))+B},$$ where $B$ is a constant which does not depend on $s\in \mathcal{S}_{\delta}.$
Taking the product over arcs in $Y_1$ and $Y_2,$ we get that $$\sign(W(s))=(-1)^{\deg(W(s))+B'},$$ for all $s\in \mathcal{S}_{\delta}$, where $B'$ is a constant which does not depend on $s.$ Therefore the maximal degree terms do not cancel out and $\underset{s\in \mathcal{S}}{\sum} W(s)\neq 0.$
    
\end{proof}

We are now ready to complete the proof of Proposition \ref{prop:nozeropoly} which as explained in the beginning  of the section
also proves Theorem \ref{thm:lower}. Recall from Section \ref{sec:two} that the completion of the proof of Theorem \ref{thm:lower}
also completes the proof of Theorem \ref{thm:inequalitiesgeneral}.

\begin{proof}[Proof of Proposition  \ref{prop:nozeropoly}]

By Proposition \ref{prop:statesum}  and Lemma \ref{lem:maxdegree} we have
$$ C(\delta)=\varepsilon(\gamma)\underset{s\in \mathcal{S}_{\delta}}  {\sum} W(s)\neq 0,$$
and by Proposition \ref{prop:polynomial} we have  $P_{\delta}(z)\neq 0$ in ${\mathbb Q}[z^{\pm1}]$.
\end{proof}


 \section{A metric from TQFT on the pants graph} \label{sec:TQFTmetric}
 Let $\Sigma$ be a closed surface of genus $g>1$.
The {\emph{pants graph}}  $\CE$ is an abstract graph whose  set of vertices  $\CV$ are in one-to-one correspondence
 with pants decompositions of $\Sigma$.  Two vertices are connected by an edge if they are related  by an $A$-move or an $S$-move shown in  Figure \ref{fig:ASmoves}. It is known that the  graph $\CE$  is connected \cite{HT}.
 
The path metric 
$d_{\pi} : \CV \times \CV \rightarrow [0,\infty),$  assigns to a pair of pants decompositions $P, Q$ the minimum number of edges over all the paths in $\CE$ from $P$ to $Q$. In this metric the length of each edge of $\CV$ has length 1.
It is known that the mapping class group acts by isometries on $(\CE, d_{\pi})$.

In the setting of Proposition \ref{prop:metricproperties} of Appendix A, $d_{\pi} $ is the metrification of the function
$${\pi}: \CV \times \CV \rightarrow [0,\infty],$$ that is defined by $\pi(P,Q)=1$ if $P$ and $Q$ are connected by an edge in
$\CE$ and $\pi(P,Q)=\infty$  otherwise.

\subsection{The TQFT  metric}We consider
$\dn : \CV\times \CV \rightarrow [0,\infty)$, the metrification of 
$n :\CV \times \CV \rightarrow [0,\infty)$,where $n(Q, P)$ is the quantum intersection number of $Q$ with respect to $P$.
More specifically, given $P, Q\in \CV$,  set  $\nu(P, Q):= \min (n(P, Q), \ n(Q, P))$ and
let $S(P,Q)$ denote the set of all  sequences $x_n:=(P_1,\ldots , P_n)$ of points $\CV$, where $n\geq 0$ is an integer, and $P_1=P$ and $P_n=Q$.
Now we define
\begin{equation}
\dn(P, Q)=\underset{x_n \in S(P, Q)}{\inf} \left( \nu(P_1, P_2)+ \ldots +  \nu(P_{n-1}, P_n) \right).
\label{eq:metricdefinition}
\end{equation}
We stress that in the above, we do not require $P_i$ and $P_{i+1}$ to be related by an elementary move.

The following lemma follows immediately from the definition.
\begin{lemma}\label{lem:bound} We have $\nu(P, Q)\geq \dn(P, Q)$.
\end{lemma}

Next we prove that $\dn$ is a metric on the $0$-skeleton of the pants graph.

\begin{lemma} \label{lem:itsmetric}The function $\dn$ is a metric and the mapping class group $\M$ acts by isometries on $(\CV, \dn)$. That is,
\begin{itemize}
 \item[(i)] we have $\dn(P, Q)=0$ if and only if  $P=Q$; and
\item[(ii)] for any $P,Q\in \CV$ and $\phi\in\M$, $\dn(\phi(P), \phi(Q))=\dn(P, Q)$.
\end{itemize}
\end{lemma}
\begin{proof} The first statement follows from Proposition \ref{prop:metricproperties} and the fact that 
$n(P, Q)\geq 2$ whenever $P\neq Q$ which follows  from Proposition \ref{prop:ASmovecomputation}.
The second statement follows from the fact that $n(\phi(P), \phi(Q))=n(P, Q)$ for $\phi\in \M$ and $P,Q\in \CV$, which was shown in Proposition \ref{prop:MCGaction}.
\end{proof}

\begin{definition}\label{def:qusiisometry} A function between metric spaces $f: (X_1, d_1) \longrightarrow (X_2, d_2)$ is called a {\emph{quasi-isometric embedding}} if there are constants $A, B$ such that
$${\frac{1}{A}}d_1(x,y)-B\leq d_2(f(x), f(y)) \leq A d_1(x,y) +B,$$
for any $x,y\in X_1$.

If $f$ has uniformly dense image we will say that $f$ is a {\emph{quasi-isometry}}, and that the spaces  $(X_1, d_1)$ and $(X_2, d_2)$ are {\emph{quasi-isometric}}.

In the case where $X_1=X_2$ we will say that the metrics $d_1$, $d_2$ are quasi-isometric if the identity is a quasi-isometric embedding.
\end{definition}

\subsection{Comparing to the path metric}The main result in this section is the following theorem that shows that the metrics $\dn(P,Q) $ and $d_{\pi}(P,Q)$ are quasi-isometric. Theorem \ref{thm:PathTQFT}, combined with work of Brock \cite{Brockpants, Brockvol}, allows us to relate
 $\dn$ to the Weil-Petersson metric  on the
Teichm\"uller space and to hyperbolic 3-dimensional geometry.

\begin{theorem} \label{thm:PathTQFT} The identity map $(\CV, \ d_{\pi})\longrightarrow (\CV, \ \dn)$ is a quasi-isometry. More specifically there is a constant $A>0$,  only depending on the topology of $\Sigma$,
such that
$$\frac{A}{3g-3}d_{\pi}(P, Q)\leq \dn(P,Q) \leq 2 d_{\pi}(P,  Q),$$
for any $P, Q\in \CV$ 
\end{theorem}

The upper bound in Theorem \ref{thm:PathTQFT} is shown in the next lemma.

\begin{lemma}\label{lem:upper} For any $P, Q\in \CV$, we have $\dn(P,Q) \leq 2 d_{\pi}(P,  Q)$.
\end{lemma}
\begin{proof}
By Proposition \ref{prop:ASmovecomputation},
 if two pairs of pants $P, P'\in \CV$ are related
 an $S$-move or an $A$-move then $n(P, P')=2$. Thus, $\dn(P, P')= \nu(P, P'):= \min (n(P, P'), \ n(P', P))\leq 2$.
 Thus if $d_{\pi}(P, P')=1$, then $\dn(P, P')\leq 2$ and we have
\begin{equation}
\dn(P, P')= 2 \ d_{\pi}(P, P').
\label{equalone}
\end{equation}
For arbitrary  $P, Q\in \CV$, pick a minimum length path from $P$ to $Q$ in $\CE$
and write $d_{\pi}(P, Q)=\sum_{i=1}^{r}d_{\pi}(P_i, P_{i+1}),$ with $d_{\pi}(P_i, P_{i+1})=1$. Now 
by \eqref{equalone}
$$2d_{\pi}(P, Q)=\sum_{i=1}^{r}2d_{\pi}(P_i, P_{i+1})=\sum_{i=1}^{r}\dn(P_i, P_{i+1})\geq \dn(P, Q),$$
which finishes the proof.

\end{proof}

\subsection{The lower bound} 
To prove  the lower inequality of \ref{thm:PathTQFT} we need the following. 

\begin{proposition} \label{pro:pathintersection} Given $\Sigma$ a closed surface of genus at least two, there is a constant $B>0$, only depending on $\s$, so that for any  $P, Q\in \CV$ we have
$\ d_{\pi}(P, Q)\leq  B\ I(P,Q)$.
\end{proposition}

Proposition \ref{pro:pathintersection} is well known to experts 
on curve complexes of surfaces and their relations to hyperbolic
geometry. A proof is implicit in Brock's work \cite{Brockpants}: The proof
of \cite[Lemma 3.3]{Brockpants} implies that there is $B>0$ so that the statement of Proposition
 \ref{pro:pathintersection} holds. The proof that $B$ only depends on $\s$ is given in the proofs of  \cite[Lemma 4.2]{Brockpants} 
 using a result
of Mazur and Minsky
\cite[Theorem 6.12]{MaMi} implying that the path distance in the pants complex of a surface is bounded above by a sum of distances in the curve complexes of (non-annular)
subsurfaces.
Aougab, Taylor and Webb \cite{ ATW} have made some of the estimates of \cite{Brockpants, MaMi} effective and 
and  \cite[Lemma 7.1]{ ATW} states 
 an effective version of Proposition \ref{pro:pathintersection} 
 making the constant $B$ explicit. Finally, Proposition \ref{pro:pathintersection} can also be derived by \cite[Theorem 6.12]{MaMi} and  a result of  Choi and Rafi \cite[Corollary D]{CRa}.
  
We are now ready to give the proof of Theorem \ref{thm:PathTQFT}.

\begin{proof} By Lemma \ref{lem:upper} we only need to deduce the lower inequality.

By definition of $\dn$, given $P, Q\in \CV$ we can find a sequence of points  $\{P_i\}_{i=1}^r$  in $\CV$
such that we have
\begin{enumerate}
\item[(i)] $P=P_1$ and $Q=P_r$,
\item[(ii)]  $\dn(P_i, P_{i+1})=\nu(P_i, P_{i+1}):={\min}\{n(P_i, P_{i+1}), \ n(P_{i+1}, P_{i})\}$, and
\item[(iii)] $\dn(P, Q)= \sum_{i=1}^{r}\dn(P_i, P_{i+1}).$
\end{enumerate}

By Theorem \ref{thm:inequalitiesgeneral}, for $1\leq i\leq r$, we have
$$\dn(P_i, P_{i+1})=\nu(P_i, P_{i+1})\geq \frac{1}{(3g-3)} I(P_i, P_{i+1}),$$

which by
Proposition \ref{pro:pathintersection} gives
$$\dn(P_i, P_{i+1})\geq \frac{1}{B(3g-3)}d_{\pi}(P_i, P_{i+1}).$$
for $i=1,\ldots, r$. Thus we obtain
$$\dn(P, Q)= \sum_{i=1}^{r}\dn(P_i, P_{i+1})\geq  \frac{1}{B(3g-3)} \sum_{i=1}^{r}d_{\pi}(P_i, P_{i+1}) \geq \frac{A}{(3g-3)}
 d_{\pi}(P, Q),$$
 where $A:=\frac{1}{B}$, and the desired result follows.
  
\end{proof}

\section{ Relations with Teichm\"uller space and hyperbolic geometry} \label{sec:TSconnections}
\subsection{ Weil-Petersson metric and TQFT} The  \emph {Teichm\"uller space}
  ${\TS}$ parametrizes finite area hyperbolic structures on $\s$:
  Points in $\TS$ are equivalence classes of pairs  $(X, \phi)$ of a finite area hyperbolic surface and a homomorphism
  $\phi: X \longrightarrow \s$, where  $(X, \phi)$ is equivalent to  $(Y, \psi)$ if there is an isometry $f:X \longrightarrow Y$  
  so that $\psi\circ f$ and $\phi$ are isotopic.

 Brock \cite{Brockpants} proved that  the 0-skeleton $\CV$ of the pants graph is quasi-isometric to
  $\TS$  with the Weil-Petersson  metric $d_{WP}$. We will use this result and 
  Theorem \ref{thm:PathTQFT}  to derive Theorem \ref{thm:QI}.
  We will not need a detailed definition of the Weil-Petersson  metric and we will
 not work with its geometric properties. Nevertheless, to convey the nature of the quasi-isometry of Theorem \ref{thm:QI}
  we will briefly discuss
 some features of the metric $d_{WP}$. For details the reader is referred to \cite{Brockpants} and references provided therein.
 
 Let ${\mathbb H}^2=\{z\in {\mathbb C}, \  {\rm Im}(z)>0 \}$ denote the upper-half complex plane with the \emph{Poincar\' e metric}.
By  uniformization,  every $X\in \TS$ becomes  an 1-dimensional complex manifold  (a Riemann surface) as quotient of ${\mathbb H}^2$  by a Fuchsian group.
The space $\TS$ has the structure of a $3g-3$ complex manifold.

The space of holomorphic quadratic differential forms of type $\phi(z) dz^2$ on $X$, denoted by $Q(X)$, has a natural identification with the holomorphic co-tangent 
space of $\TS$ at $X$. One can define a Hermitian inner product on $Q(X)$ by integrating over $X$, using the hyperbolic metric of $X$ which induces a Hermitian inner product on the tangent space at $X$. Since we have such an inner product for each $X\in \TS$ we obtain a Riemannian metric,
the Weil-Petersson. Let $\dwp:  \TS \times \TS\longrightarrow [0, \infty)$ denote the associated distance of this metric.
It is known that $\M$ acts by isometries on $(\TS, \dwp)$.

By a result of Bers  there is a constant $L>0$ (Ber's constant)
only depending on the surface $\s$, so that given $X\in \TS$ there is a pants decomposition $P_X\in \CV$ such that
for each curve $\gamma \in P_X$ the length  of $\gamma$ in the hyperbolic metric of $X$, say $l_X(\gamma)$, is less that $L$.
This implies that the sets $\{V(P)\}_{P\in \CV}$, where
\begin{equation}
V(P):=\{ X\in \TS\  | \ \underset{\gamma\in P} {\max} \{l_X(\gamma)\}<L\},
\label{levelsets}
\end{equation}
cover the Teichm\"uller space $\TS$.

\begin{theorem} \label{thm:pathPW}{ \rm (\cite[Theorem 3.2]{Brockpants})}  An embedding $g: \CV \longrightarrow \TS$, with  $g(P)\in V(P)$, for all $P\in\CV$,  is a quasi-isometry of 
$(\CV, \ d_{\pi})$
to $(\TS, \ \dwp)$.

\end{theorem}

Theorem \ref{thm:PathTQFT}, implies that any embedding $g: \CV \longrightarrow \TS$ as in Theorem \ref{thm:pathPW}, is also quasi-isometry of $\CV$
with the quantum metric $\dn$
 to  $\TS$ with $ \dwp$. Hence we have:
\begin{theorem} \label{thm:WPQT} There are constants $A_1, A_2>0$, only depending on $\s$,
such that 

$${\frac{1}{A_1}}\dn(P_X, P_Y)-A_2\leq \dwp(X, Y) \leq A_1 \dn(P_X, \  P_Y) +A_2,$$
for any $X, Y \in \TS$.
\end{theorem}

\subsection{Relations to volumes of fibered 3-manifolds} Here we discuss relations between the quantum metric $\dn$ and hyperbolic volumes of 3-manifolds.
To state our result  we need the following definition.

\begin{definition}\label{tl}
Given a metric $(X, d)$ and an isometry $\phi: X\longrightarrow X$, the {\emph{translation length}} 
is defined by
$$L_d(\phi)={\rm{inf}} \{ d(\phi(x), x)\ \ | x\in X\}.$$

\end{definition}

Since $\M$ acts by isometries on $(\CE, \dn),$ given
$\phi \in \M$ we can define the translation length
$L^{\dn}(\phi)$. Similarly, we can define
$L^{d_{\pi}}(\phi)$. Also since $\M$ is known to act by isometries on  $({\mathcal T}, d_{WP}),$ and we can define $L^{d_{WP}}(\phi)$.

Given $\phi \in \mathrm{Mod}(\Sigma)$, 
let $M_{\phi}=F \times [0,1]/_{(x,1)\sim ({ {\phi}}(x),0)}$ be the mapping torus of $\phi$. 
By work of Thurston \cite{thurston:hypII}, a mapping class $\phi \in \M$ is pseudo-Anosov  if and only if  the  mapping torus $M_{\phi}$ is a hyperbolic 3-manifold.
By Mostow rigidity, the hyperbolic metric 
is essentially unique and the hyperbolic volume $\vol(M_{\phi})$ is a topological invariant of $M_{\phi}$.

\begin{named}{Theorem \ref{thm:volume}}
There exist  a positive  constant $N$, depending only on the topology of the surface $\s$,  so that for any  pseudo-Anosov mapping class $\phi \in \M$
 we have
 
 $${\frac{1}{N}}\ L^{\dn}(\phi)\leq \vol(M_{\phi}) \leq N\   L^{\dn},$$
 \end{named} 
 
  \begin{proof} 
  A result of Brock \cite[Theorem 1.1]{Brockvol} states that  there is a constant  $k>0$ only depending  on the topology of $\s$ so that
  \begin{equation}
 {\frac{1}{k}\ }L^{d_{WP}}(\phi)\leq \vol(M_{\phi}) \leq k \ L^{d_{WP}}(\phi).
 \label{WPV}
 \end{equation}
   By Theorem \ref{thm:WPQT} the space
   $(\CV, \dn)$ is quasi-isometric to $({\mathcal T}, \dwp)$. Thus, after re-adjusting the constants of the quasi-isometry and Definition \ref{tl}, we have 
   $ {\frac{1}{M}}L^{\dn}(\phi)\leq L^{d_{WP}}(\phi) \leq M L^{\dn}(\phi)$, which combined with
 \eqref{WPV} gives the result.
 \end{proof}

 To conclude this section, we remark that there exist several open conjectures  predicting strong relations of asymptotic aspects of the $SU(2)$-Witten-Reshetikhin-Turaev TQFT
 with hyperbolic geometry, and partial verifications of these conjectures \cite{MR2797089, yangsurvey}.
 
 \subsection{Nielsen-Thurston type detection} 
 By the  Nielsen-Thurston classification mapping classes are divided into three types: \emph{periodic, reducible} and
\emph{pseudo-Anosov}.  Furthermore, the type of $f$ determines the geometric structure, in the sense of Thurston, of the 3-manifold obtained as mapping torus of $f.$
The AMU conjecture, due to Andersen, Masbaum and Ueno  \cite{AMU} asserts that the   large-$r$ asymptotic behavior of 
$SU(2)$-quantum mapping class group representations $\{\rho_r\}_{r\geq 3}$ detects pseudo-Anosov elements. For example, it states that
for any  pseudo-Anosov mapping class $\phi\in \M$, the image $\rho_r(\phi)$ has infinite order for all but finitely many $r.$ Furthermore, in the case of the $4$-holed sphere \cite{AMU} (or the punctured torus in \cite{San12}), it has been observed that the spectral radii of the images $\{\rho_r\}_{r\geq 3}$
also determine the \emph{stretch factor} of $\phi$. 
 
 Our results in this subsection are, in the broad sense of things, similar in spirit to the AMU conjecture. The first result, which is  Corollary \ref{cor:content} stated in the introduction,
sheds some light on the geometric content of the complexity $n(P, \ \phi(P))$ for
 pseudo-Anosov $\phi \in \M$. 
   
  \begin{named}{Corollary \ref{cor:content}}There  is a constant $N>0$, only depending  on the topology of $\s$, so that for any pseudo-Anosov  mapping class $\phi \in \M$ and any $P\in \CV$ we have
$$n(P, \phi(P)) \geq N \  \vol(M_{\phi}).$$
 \end{named}
\begin{proof} By Theorem \ref{thm:volume} and the definition of $L^{\dn}(\phi)$, there is a positive constant $N$ only depending  on the topology of $\s$, such that
$$\dn(P, \ \phi(P))\geq  N \ L^{\dn}(\phi),$$
for any $P\in \CV$. 
On the other hand, by Lemma \ref{lem:bound}, we have   $$\nu(P, \phi(P))\geq \dn(P, \phi(P)),$$ 
where $\nu(P, \phi(P))= \min ( n( P, \phi(P)), \  n(\phi(P), P))$, and the desired result follows.
\end{proof}

The next result uses Theorem
\ref{thm:inequalitiesgeneral} to derive a characterization of pseudo-Anosov mapping classes in terms of the behavior of quantum intersection numbers
under iteration in $\M$.

\begin{corollary} \label{cor:detect} A mapping class $\phi\in \M$ is  pseudo-Anosov if and only if for any multicurve  $\gamma$, 
we have $ \underset{k\to\infty}{\lim}n(\phi^{k}(\gamma),P)=\infty$, for all $P\in
 \CV$.
\end{corollary}
\begin{proof} By Theorem \ref{thm:inequalitiesgeneral},  for any $\gamma$  we have $n(\phi^{k}(\gamma), P)\geq \frac{I(\phi^{k}(\gamma),P)}{3g-3}$, for all $P\in \CV$.
If $\phi$ is  pseudo-Anosov, it is known that $ \underset{k\to\infty}{\lim}I(\phi^{k}(\gamma),P)=\infty$ by \cite[Theorem 12.2]{FLP}. 

For the converse, we argue that if $\phi$  is not pseudo-Anosov, there is a  multicurve  $\gamma$ such that, for any $P\in \CV$, the sequence
$\{n(\phi^{k}(\gamma), P)\}_{k\geq 1}$ is bounded.
To that end,  recall that, by definition, if $\phi$  is not pseudo-Anosov, then there is a  multicurve  $\gamma$, and an integer $k_1>0$, such that
$\phi^{k_1}$ leaves $\gamma$ fixed pointwise. It follows that, for any $0<k\in \NN$,   $n(\phi^{k}(\gamma), P)\leq B<\infty$, where $B:=\underset{1\leq k\leq k_1}{\max}n(\phi^{k}(\gamma), P)$.
\end{proof}

\begin{remark} {\rm In fact, it is known that if $\phi\in \M$ is not periodic then we can find a simple closed curve $\gamma\subset \s$ so that $I(\phi^{k}(\gamma), \gamma)$ increases with $k$.
Completing $\gamma$ to a pants decomposition and applying
Theorem \ref{thm:inequalitiesgeneral}, $ \underset{k\to\infty}{\lim}n(\phi^{k}(\gamma),P)=\infty$. Thus, we can characterize periodic mapping classes as follows:
$\phi$ is periodic  if and only if for any multicurve  $\gamma$, 
we have $ \underset{k\to\infty}{\lim}n(\phi^{k}(\gamma),P)<\infty$, for all $P\in
 \CV$.}

\end{remark}

We continue with the following corollary:

\begin{corollary}\label{cor:stretch} Let $\phi\in \M$ be a  pseudo-Anosov mapping class with stretch factor $\lambda_{\phi}$. 
Then, for any  $P\in \CV$
and  any simple closed  curve $\gamma$, we have
 
 $$\lambda_{\phi} \: \leq \: \underset{k\to\infty}{\liminf} \sqrt[k]{n(\phi^{k}(\gamma),P)} \: \leq \underset{k\to\infty}{\limsup} \sqrt[k]{n(\phi^{k}(\gamma),P)} \: \leq \: \lambda_{\phi}^{3g-3}.$$
\end{corollary}

 \begin{proof} By  \cite[Theorem 12.2]{FLP}, for any simple closed curve $\beta$ we have
 \begin{equation}\label{eq:fm}
 \underset{k\to\infty}{\lim} \sqrt[k]{I(\phi^{k}(\gamma),\beta)}=\lambda_{\phi}.
 \end{equation}
  
Given $\gamma$ and $P$, for any $k>0$,  pick  $\alpha(k) \in P$ so that   $ \underset{\alpha \in P}{\max}\{I(\phi^{k}(\gamma), \alpha)\}=I(\phi^{k}(\gamma), \alpha(k))$.
Since $ I(\phi^{k}(\gamma), \alpha(k))\leq I(\phi^{k}(\gamma),P)\leq (3g-3) I(\phi^{k}(\gamma), \alpha(k)),$
Theorem
\ref{thm:inequalitiesgeneral} gives

\begin{equation}\label{eq:fm1}
\frac{I(\phi^{k}(\gamma),  \alpha(k))}{3g-3}\leq  n(\phi^{k}(\gamma),P)  \leq  (6g-6)^{3g-3} \  I(\phi^{k}(\gamma), \alpha(k))^{3g-3}.
 \end{equation}
 We can split $\NN$ into finitely many  subsequences so that for each of them, say $\bf{s}$, for all $k\in \bf{s}$ the curve $ \alpha(k)$ of \eqref{eq:fm1} is the same curve $\beta \in P$.
 Then, using \eqref{eq:fm} and \eqref{eq:fm1}, for each of the subsequences the result follows.
 \end{proof}

 We close the section with the proof of Corollary \ref{amu} which we restate here.

\begin{named} {Corollary \ref{amu}}
Let $\phi$ be a  pseudo-Anosov mapping class such that there is  a multicurve $\gamma$ and $P\in \CV$
so that  $\underset{m\to\infty}{\limsup} \sqrt[m]{r_0(\phi^{m}(\gamma),P)} =l< \infty$. Suppose that $\phi$ fails the AMU conjecture. That is, there exists
an infinite sequence of positive integers $\{m_r\}_{r\to \infty}$ such that
$\rho_r((\phi)^{m_r})=1$.  Then,  $\underset{m_r\to\infty}{\limsup} \sqrt[m_r]{r} <l$.
\end{named}
\begin{proof}
 Suppose 
 there is an infinite sequence of positive integers $\{m_r\}_{r\to \infty}$ such that
$\rho_r((\phi)^{m_r})=1$. As in the proof of Proposition \ref{prop:MCGaction}, 
$T_r^{{\phi^{m_r}}(\gamma)}=T_r^{\gamma}$.   This implies that as $r\to \infty$, the matrix $A(\phi,\ m_r)$, representing  $T_r^{{\phi^{m_r}}(\gamma)}$ in the basis ${\mathcal B}_P$, is equal to the matrix  $A(\gamma,r)$ representing $T_r^{\gamma}$. For $r$ large enough, the matrix $A(\gamma,r)$
has $n(\gamma, P)$ non main diagonals containing non-zero entries.
Since $\phi$ is pseudo-Anosov,  by Corollary \ref{cor:detect}, the quantum intersection number $n(\phi^{m_r}(\gamma), P)$ increases as $m_r$ does.
Thus, for $r$ large enough  $A(\phi,\ m_r)$ can equal $A(\gamma,r)$ only if
$r< r_0(\phi^{m_r}(\gamma), P)$, which implies the conclusion. \end{proof}


\section{Norms of curve operators and quantum intersection number} 
\label{sec:norms}

Let the notation and setting be as in Section \ref{sec:two}. Let $\gamma$ be a multicurve and $P$ a pants decomposition
of $\s$.  Consider the curve operators $\{T_r^{\gamma}\}_{r\geq 3}$, and the orthonormal basis ${\mathcal B}_P=\{\phi_{{\bf c}} \ | \ {\bf c}\in U_r\}$ of  the Hermitian spaces $V_r(\Sigma)$ corresponding to $P$.
Let $m_r(\gamma, P)$ denote the number of non-zero entries in the matrix representing $T_r^{\gamma}$ with respect to the basis ${\mathcal B}_P$. Recall that
the quantum  intersection number $n(\gamma, P)$ denotes the number of non-main diagonals containing non-zero entries and this is captured by the number of non-zero functions
$G_{\bf k}^{P}(\frac{c}{r},\frac{1}{r})$ in Equation \eqref{eq:opcoefficients}. We have the following.

\begin{proposition}\label{prop:matrixEntries} For any multicurve $\gamma$ and pants decomposition $P$, for $r$ large enough, 

\begin{equation} \label{eq:mn}
O\big( v_r \ n(\gamma,P)\big)\leq m_r(\gamma, P)\leq v_r\left(n(\gamma, P)+1\right),
\end{equation}
where $v_r:=\dim {V_r(\Sigma)}$ is given by the Verlinde formula \cite{BHMV2}.

\end{proposition}
\begin{proof} For any $r\geq 3$, the  upper  inequality follows by  the definitions, where the term $+1$ accounts for the main diagonal, which was discarded in the count for $n(\gamma,P).$ 

There are $n(\gamma,P)$ non-zero functions $G_{\bf k}^P$, and furthermore those functions are analytic on a subset $V_{\gamma} \subset U\times [0,1],$ where $U$ is the open set defined in Section \ref{sec:two}. Since they are non-zero and analytic, we can pick an open subset of $V_{\gamma}$ of the form $V\times [0,h)$ on which all of the functions $G_{\bf k}^P$ are non-vanishing. Then we have that $m_r(\gamma,P)\geq |\frac{1}{r}U_r \cap V|,$ which grows like $r^{3g-3},$ the same as $v_r,$ since it amounts to counting a number of lattice points in the open set $rV.$
\end{proof}
It follows that for $r$ large enough, the quantity $ m_r(\gamma, P)$ admits two-sided bounds in terms of the geometric intersection number $I(\gamma, P)$ and  $v_r.$ 

Next we consider  the $l^1$ and $l^2$-norms of $T_r^{\gamma}$, with respect to ${\mathcal B}_P$,  defined by
$$||T_r^{\gamma}||_{\{l^s, P\}}=\big( \underset{{\bf c}, {\bf d}\in U_r}{\sum}|\langle T_r^{\gamma}\phi_{\bf c} ,\phi_{\bf d} \rangle|^2\big)^{1\over s},$$
for $s=1,2$.
\begin{lemma}\label{lemma:l2norm}The norm $||T_r^{\gamma}||_{\lbrace l^2,P\rbrace}$ is independent of the pants decomposition $P$ and depends only on the orbit of $\gamma$ under the action of $\M$.
\end{lemma}
\begin{proof}
Recall that the $l^2$-norm can also be computed as $||T_r^{\gamma}||_{\lbrace l^2,P\rbrace}=\mathrm{Tr}(T_r^{\gamma}(T_r^{\gamma})^*)^{\frac{1}{2}},$ where $(T_r^{\gamma})^*$ is the adjoint for the natural Hermitian norm on $V_r(\Sigma).$ However, $T_r^{\gamma}$ is an Hermitian operator, hence 
$$\mathrm{Tr}(T_r^{\gamma}(T_r^{\gamma})^*)=\mathrm{Tr}((T_r^{\gamma})^2)=\underset{\lambda \in \mathrm{\sigma(T_r^{\gamma})}}{\sum} |\lambda|^2,$$
where $\sigma(T_r^{\gamma})$ is the set of eigenvalues of $T_r^{\gamma}.$
This expression shows that the $l^2$-norm is independent of $P,$ since a trace is independent of basis, and the $l^2$-norm depends only on the orbit of $\gamma$ under $\mathrm{Mod}(\Sigma)$ since $T_r^{\phi(\gamma)}=\rho_r(\phi)\circ T_r^{\gamma} \circ \rho_r(\phi)^{-1}$ for $\phi \in \mathrm{Mod}(\Sigma).$
\end{proof}
Turning to the  $l^1$-norm, an easy application of
the Cauchy-Schwarz inequality gives
$||T_r^{\gamma}||_{\{l^1, P\}}\leq (m(\gamma, P))^{1/2} ||T_r^{\gamma}||_{l^2}$, which combined with the last inequality and  \eqref{eq:mn},
gives
$$\frac{||T_r^{\gamma}||_{\{ l^1, P\}}}{||T_r^{\gamma}||_{l^2}}\leq (n(\gamma, P)+1)^{1/2} \ v_r^{1/2},$$
and
 \begin{equation}\label{eq:Tn}
 T(\gamma, P):=\underset{r\to \infty}{\limsup} \frac{{||T_r^{\gamma}||_{\{l^1, P\}}}}
{||T_r^{\gamma}||_{l^2}v_r^{1/2}}\leq (n(\gamma, P)+1)^{1/2}.
\end{equation}
Hence, information about the geometric content of $n(\gamma, P)$ translates into information on the geometric content of the $T(\gamma, P)$ that encodes the asymptotic behavior of
the curve operator norms. For example, combining \eqref{eq:Tn} with Corollary
\ref{cor:stretch} we get,
\begin{corollary}
Let $\phi\in \M$ be a  pseudo-Anosov mapping class with stretch factor $\lambda_{\phi}$. 
Then, for any  $P\in \CV$
and  any simple closed  curve $\gamma$, we have
 
 $$ \underset{k\to\infty}{\limsup} \sqrt[k]{T(\phi^{k}(\gamma),P)} \: \leq \: \lambda_{\phi}^{\frac{3g-3}{2}}.$$
\end{corollary}

As another example,  in the case that $\gamma$ is also a pants decomposition, one can directly compare the function $T$ to the metric $\dn$ by means of Proposition \ref{prop:metricproperties}.

Building on our work in Sections \ref{sec:five} and
 \ref{sec:fusion_computations}, in a subsequent paper, 
we will curry out a more detailed  study of the asymptotics of the coefficients $G_{\bf k}^{\gamma}(\frac{c}{r},\frac{1}{r})$ to derive a lower bound of $T(\gamma, P)$ in terms of $n(P, Q)$.


\section{Proofs of auxiliary lemmas} \label{section:auxiliary}

We give the proofs of Lemmas \ref{lemma:sliding} and \ref{lemma:singleArcAnnuli}. In the figures, thin black edges are colored by $2$, while the colors of red edges are specified and the fusion rules applied are the ones of Figure \ref{fig:fusionRules2}.

\begin{named}{Lemma \ref{lemma:sliding}}{{\rm (}\emph{Sliding  Lemma}{\rm )}} With the conventions of Figure \ref{fig:fusionRules2}, and for any $\varepsilon,\mu\in \lbrace \pm 1 \rbrace,$
we have the following identities.
\vskip 0.05in
\begin{center}
\def \svgwidth{.9\columnwidth}

\end{center} 
\end{named}
\begin{proof}
Using the strand rule, then the triangle rule, we get
\begin{center}
\def \svgwidth{.78\columnwidth}
\begingroup%
  \makeatletter%
  \providecommand\color[2][]{%
    \errmessage{(Inkscape) Color is used for the text in Inkscape, but the package 'color.sty' is not loaded}%
    \renewcommand\color[2][]{}%
  }%
  \providecommand\transparent[1]{%
    \errmessage{(Inkscape) Transparency is used (non-zero) for the text in Inkscape, but the package 'transparent.sty' is not loaded}%
    \renewcommand\transparent[1]{}%
  }%
  \providecommand\rotatebox[2]{#2}%
  \newcommand*\fsize{\dimexpr\f@size pt\relax}%
  \newcommand*\lineheight[1]{\fontsize{\fsize}{#1\fsize}\selectfont}%
  \ifx\svgwidth\undefined%
    \setlength{\unitlength}{445.66546102bp}%
    \ifx\svgscale\undefined%
      \relax%
    \else%
      \setlength{\unitlength}{\unitlength * \real{\svgscale}}%
    \fi%
  \else%
    \setlength{\unitlength}{\svgwidth}%
  \fi%
  \global\let\svgwidth\undefined%
  \global\let\svgscale\undefined%
  \makeatother%
  \begin{picture}(1,0.42866575)%
    \lineheight{1}%
    \setlength\tabcolsep{0pt}%
    \put(0,0){\includegraphics[width=\unitlength,page=1]{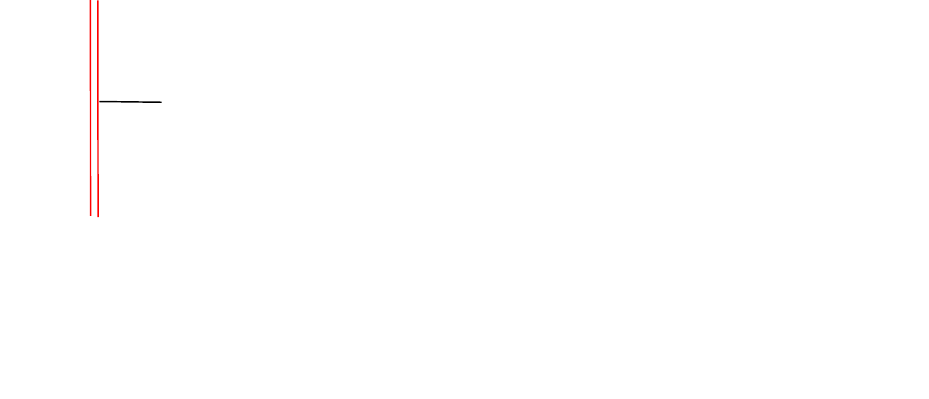}}%
    \put(0.11980607,0.20755644){\color[rgb]{0,0,0}\makebox(0,0)[lt]{\lineheight{1.25}\smash{\begin{tabular}[t]{l}$a$\end{tabular}}}}%
    \put(0.11743355,0.26445698){\color[rgb]{0,0,0}\makebox(0,0)[lt]{\lineheight{1.25}\smash{\begin{tabular}[t]{l}$a+\varepsilon$\end{tabular}}}}%
    \put(0.11683886,0.33949028){\color[rgb]{0,0,0}\makebox(0,0)[lt]{\lineheight{1.25}\smash{\begin{tabular}[t]{l}$a+\varepsilon+\mu$\end{tabular}}}}%
    \put(0.2232171,0.25367066){\color[rgb]{0,0,0}\makebox(0,0)[lt]{\lineheight{1.25}\smash{\begin{tabular}[t]{l}$=$\end{tabular}}}}%
    \put(0,0){\includegraphics[width=\unitlength,page=2]{Slidingproof1.pdf}}%
    \put(0.3864348,0.20664457){\color[rgb]{0,0,0}\makebox(0,0)[lt]{\lineheight{1.25}\smash{\begin{tabular}[t]{l}$a$\end{tabular}}}}%
    \put(0.38406232,0.26354512){\color[rgb]{0,0,0}\makebox(0,0)[lt]{\lineheight{1.25}\smash{\begin{tabular}[t]{l}$a+\varepsilon$\end{tabular}}}}%
    \put(0,0){\includegraphics[width=\unitlength,page=3]{Slidingproof1.pdf}}%
    \put(0.38151842,0.32959007){\color[rgb]{0,0,0}\makebox(0,0)[lt]{\lineheight{1.25}\smash{\begin{tabular}[t]{l}$a+\varepsilon+\mu$\end{tabular}}}}%
    \put(0.37851329,0.40514443){\color[rgb]{0,0,0}\makebox(0,0)[lt]{\lineheight{1.25}\smash{\begin{tabular}[t]{l}$a+\varepsilon+\mu$\end{tabular}}}}%
    \put(0.38031634,0.36814884){\color[rgb]{0,0,0}\makebox(0,0)[lt]{\lineheight{1.25}\smash{\begin{tabular}[t]{l}$a+\varepsilon+\mu+1$\end{tabular}}}}%
    \put(0,0){\includegraphics[width=\unitlength,page=4]{Slidingproof1.pdf}}%
    \put(0.7085355,0.20651431){\color[rgb]{0,0,0}\makebox(0,0)[lt]{\lineheight{1.25}\smash{\begin{tabular}[t]{l}$a$\end{tabular}}}}%
    \put(0.70616292,0.26341486){\color[rgb]{0,0,0}\makebox(0,0)[lt]{\lineheight{1.25}\smash{\begin{tabular}[t]{l}$a+\varepsilon$\end{tabular}}}}%
    \put(0,0){\includegraphics[width=\unitlength,page=5]{Slidingproof1.pdf}}%
    \put(0.70361897,0.32945979){\color[rgb]{0,0,0}\makebox(0,0)[lt]{\lineheight{1.25}\smash{\begin{tabular}[t]{l}$a+\varepsilon+\mu$\end{tabular}}}}%
    \put(0.70061384,0.40501416){\color[rgb]{0,0,0}\makebox(0,0)[lt]{\lineheight{1.25}\smash{\begin{tabular}[t]{l}$a+\varepsilon+\mu$\end{tabular}}}}%
    \put(0.70241694,0.36801858){\color[rgb]{0,0,0}\makebox(0,0)[lt]{\lineheight{1.25}\smash{\begin{tabular}[t]{l}$a+\varepsilon+\mu-1$\end{tabular}}}}%
    \put(0.5321452,0.26357089){\color[rgb]{0,0,0}\makebox(0,0)[lt]{\lineheight{1.25}\smash{\begin{tabular}[t]{l}$-$\end{tabular}}}}%
    \put(0,0){\includegraphics[width=\unitlength,page=6]{Slidingproof1.pdf}}%
    \put(0.33165137,0.01276606){\color[rgb]{0,0,0}\makebox(0,0)[lt]{\lineheight{1.25}\smash{\begin{tabular}[t]{l}$a$\end{tabular}}}}%
    \put(0.32927884,0.06966662){\color[rgb]{0,0,0}\makebox(0,0)[lt]{\lineheight{1.25}\smash{\begin{tabular}[t]{l}$a+\varepsilon+\mu+1$\end{tabular}}}}%
    \put(0.32868418,0.14469992){\color[rgb]{0,0,0}\makebox(0,0)[lt]{\lineheight{1.25}\smash{\begin{tabular}[t]{l}$a+\varepsilon+\mu$\end{tabular}}}}%
    \put(0.14927388,0.09024349){\color[rgb]{0,0,0}\makebox(0,0)[lt]{\lineheight{1.25}\smash{\begin{tabular}[t]{l}$=A$\end{tabular}}}}%
    \put(0.55844913,0.08832338){\color[rgb]{0,0,0}\makebox(0,0)[lt]{\lineheight{1.25}\smash{\begin{tabular}[t]{l}$-B$\end{tabular}}}}%
    \put(0,0){\includegraphics[width=\unitlength,page=7]{Slidingproof1.pdf}}%
    \put(0.73467995,0.01360431){\color[rgb]{0,0,0}\makebox(0,0)[lt]{\lineheight{1.25}\smash{\begin{tabular}[t]{l}$a$\end{tabular}}}}%
    \put(0.73230742,0.07050484){\color[rgb]{0,0,0}\makebox(0,0)[lt]{\lineheight{1.25}\smash{\begin{tabular}[t]{l}$a+\varepsilon+\mu-1$\end{tabular}}}}%
    \put(0.73171276,0.14553816){\color[rgb]{0,0,0}\makebox(0,0)[lt]{\lineheight{1.25}\smash{\begin{tabular}[t]{l}$a+\varepsilon+\mu$\end{tabular}}}}%
  \end{picture}%
\endgroup%

\end{center}
where $A$ and $B$ are constants that now we explain how to compute:
Note that in the identity above, passing from the first to the second  row, we have used triangle moves to eliminate a triangle from each of the  second and the third term of the first row.
However, the
 triangle rules we need here are not the ones
 depicted  in Figure \ref{fig:fusionRules2} since, for instance, the red sides of the triangles eliminated are not colored by the same color.
 Hence to compute $A$, $B$ we need to first
  to apply  the general triangle rules of Figure \ref{fig:fusionRules} to compute the coefficient function,  and then take the limit $r\rightarrow \infty,$ $\frac{a}{r}\rightarrow \theta.$

If $\varepsilon=+1,$ we get $$A=\underset{r\rightarrow\infty,\frac{a}{r}\rightarrow \theta}{\lim} \frac{\langle \frac{2+a+\varepsilon+\mu+1-(a+\varepsilon)}{2} \rangle \langle \frac{2-(a+\varepsilon+\mu+1)+a+\varepsilon}{2} \rangle}{\langle a+\varepsilon+\mu+1\rangle \langle a\rangle}=0,$$ and $$B=\underset{r\rightarrow\infty,\frac{a}{r}\rightarrow \theta}{\lim}-\frac{\langle \frac{2+a+\varepsilon+\mu+a+\varepsilon-1}{2}\rangle \langle \frac{a+\varepsilon+\mu+a+\varepsilon-2-1}{2}\rangle}{\langle a+\varepsilon+\mu-1 \rangle \langle a \rangle}=-1.$$
Since $a+\varepsilon+\mu-1=a+\mu,$ this agrees with the first identity in the statement of the lemma.. Similarly, in the case $\varepsilon=-1,$ we get $A=1$ and $B=0,$ which,  since $a+\varepsilon+\mu+1=a+\mu$ in that case,
agrees with first identity in the statement of the lemma.

Now we explain how to prove the second identity in the statement of the lemma: Using the Kauffman bracket skein rule and keeping in mind that $\underset{r\rightarrow\infty,\frac{a}{r}\rightarrow \theta}{\lim} \zeta_r=-1,$ we get:
\begin{center}
\def \svgwidth{.63\columnwidth}
\begingroup%
  \makeatletter%
  \providecommand\color[2][]{%
    \errmessage{(Inkscape) Color is used for the text in Inkscape, but the package 'color.sty' is not loaded}%
    \renewcommand\color[2][]{}%
  }%
  \providecommand\transparent[1]{%
    \errmessage{(Inkscape) Transparency is used (non-zero) for the text in Inkscape, but the package 'transparent.sty' is not loaded}%
    \renewcommand\transparent[1]{}%
  }%
  \providecommand\rotatebox[2]{#2}%
  \newcommand*\fsize{\dimexpr\f@size pt\relax}%
  \newcommand*\lineheight[1]{\fontsize{\fsize}{#1\fsize}\selectfont}%
  \ifx\svgwidth\undefined%
    \setlength{\unitlength}{656.05498547bp}%
    \ifx\svgscale\undefined%
      \relax%
    \else%
      \setlength{\unitlength}{\unitlength * \real{\svgscale}}%
    \fi%
  \else%
    \setlength{\unitlength}{\svgwidth}%
  \fi%
  \global\let\svgwidth\undefined%
  \global\let\svgscale\undefined%
  \makeatother%
  \begin{picture}(1,0.14357606)%
    \lineheight{1}%
    \setlength\tabcolsep{0pt}%
    \put(0,0){\includegraphics[width=\unitlength,page=1]{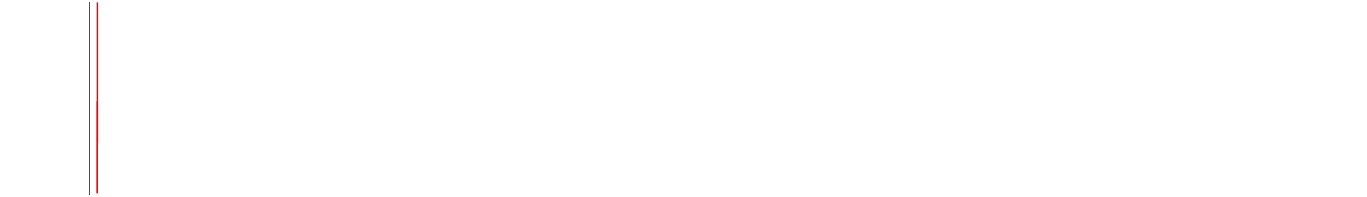}}%
    \put(0.07944731,0.00993292){\color[rgb]{0,0,0}\makebox(0,0)[lt]{\lineheight{1.25}\smash{\begin{tabular}[t]{l}$a$\end{tabular}}}}%
    \put(0.07944737,0.05402766){\color[rgb]{0,0,0}\makebox(0,0)[lt]{\lineheight{1.25}\smash{\begin{tabular}[t]{l}$a+\mu$\end{tabular}}}}%
    \put(0.07977751,0.11666206){\color[rgb]{0,0,0}\makebox(0,0)[lt]{\lineheight{1.25}\smash{\begin{tabular}[t]{l}$a+\varepsilon+\mu$\end{tabular}}}}%
    \put(0,0){\includegraphics[width=\unitlength,page=2]{Slidingproof2.pdf}}%
    \put(0.45840791,0.01108773){\color[rgb]{0,0,0}\makebox(0,0)[lt]{\lineheight{1.25}\smash{\begin{tabular}[t]{l}$a$\end{tabular}}}}%
    \put(0.45840798,0.05518246){\color[rgb]{0,0,0}\makebox(0,0)[lt]{\lineheight{1.25}\smash{\begin{tabular}[t]{l}$a+\mu$\end{tabular}}}}%
    \put(0.45873812,0.11781686){\color[rgb]{0,0,0}\makebox(0,0)[lt]{\lineheight{1.25}\smash{\begin{tabular}[t]{l}$a+\varepsilon+\mu$\end{tabular}}}}%
    \put(0.26617165,0.06832362){\color[rgb]{0,0,0}\makebox(0,0)[lt]{\lineheight{1.25}\smash{\begin{tabular}[t]{l}$=-$\end{tabular}}}}%
    \put(0,0){\includegraphics[width=\unitlength,page=3]{Slidingproof2.pdf}}%
    \put(0.68940701,0.07063323){\color[rgb]{0,0,0}\makebox(0,0)[lt]{\lineheight{1.25}\smash{\begin{tabular}[t]{l}$-$\end{tabular}}}}%
    \put(0,0){\includegraphics[width=\unitlength,page=4]{Slidingproof2.pdf}}%
    \put(0.83303589,0.00877813){\color[rgb]{0,0,0}\makebox(0,0)[lt]{\lineheight{1.25}\smash{\begin{tabular}[t]{l}$a$\end{tabular}}}}%
    \put(0.83303589,0.05287286){\color[rgb]{0,0,0}\makebox(0,0)[lt]{\lineheight{1.25}\smash{\begin{tabular}[t]{l}$a+\mu$\end{tabular}}}}%
    \put(0,0){\includegraphics[width=\unitlength,page=5]{Slidingproof2.pdf}}%
    \put(0.833366,0.11550725){\color[rgb]{0,0,0}\makebox(0,0)[lt]{\lineheight{1.25}\smash{\begin{tabular}[t]{l}$a+\varepsilon+\mu$\end{tabular}}}}%
  \end{picture}%
\endgroup%

\end{center}
Next we can simplify the first figure of the right hand side with the bigon rule. If $\varepsilon=\mu,$ then the first figure vanishes and we get the lemma. On the other hand, if $\varepsilon=-\mu$ we have
\begin{center}
\def \svgwidth{.6\columnwidth}
\begingroup%
  \makeatletter%
  \providecommand\color[2][]{%
    \errmessage{(Inkscape) Color is used for the text in Inkscape, but the package 'color.sty' is not loaded}%
    \renewcommand\color[2][]{}%
  }%
  \providecommand\transparent[1]{%
    \errmessage{(Inkscape) Transparency is used (non-zero) for the text in Inkscape, but the package 'transparent.sty' is not loaded}%
    \renewcommand\transparent[1]{}%
  }%
  \providecommand\rotatebox[2]{#2}%
  \newcommand*\fsize{\dimexpr\f@size pt\relax}%
  \newcommand*\lineheight[1]{\fontsize{\fsize}{#1\fsize}\selectfont}%
  \ifx\svgwidth\undefined%
    \setlength{\unitlength}{506.0828694bp}%
    \ifx\svgscale\undefined%
      \relax%
    \else%
      \setlength{\unitlength}{\unitlength * \real{\svgscale}}%
    \fi%
  \else%
    \setlength{\unitlength}{\svgwidth}%
  \fi%
  \global\let\svgwidth\undefined%
  \global\let\svgscale\undefined%
  \makeatother%
  \begin{picture}(1,0.1427098)%
    \lineheight{1}%
    \setlength\tabcolsep{0pt}%
    \put(0,0){\includegraphics[width=\unitlength,page=1]{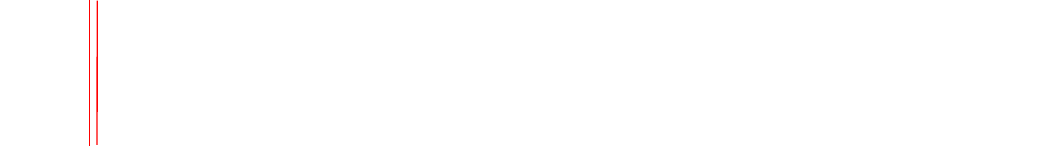}}%
    \put(0.10285982,0.01293954){\color[rgb]{0,0,0}\makebox(0,0)[lt]{\lineheight{1.25}\smash{\begin{tabular}[t]{l}$a$\end{tabular}}}}%
    \put(0.1028599,0.05612965){\color[rgb]{0,0,0}\makebox(0,0)[lt]{\lineheight{1.25}\smash{\begin{tabular}[t]{l}$a-\varepsilon$\end{tabular}}}}%
    \put(0.10328786,0.11747906){\color[rgb]{0,0,0}\makebox(0,0)[lt]{\lineheight{1.25}\smash{\begin{tabular}[t]{l}$a$\end{tabular}}}}%
    \put(0,0){\includegraphics[width=\unitlength,page=2]{Slidingproof3.pdf}}%
    \put(0.22232678,0.07542197){\color[rgb]{0,0,0}\makebox(0,0)[lt]{\lineheight{1.25}\smash{\begin{tabular}[t]{l}$=\varepsilon$\end{tabular}}}}%
    \put(0,0){\includegraphics[width=\unitlength,page=3]{Slidingproof3.pdf}}%
    \put(0.38958862,0.0910809){\color[rgb]{0,0,0}\makebox(0,0)[lt]{\lineheight{1.25}\smash{\begin{tabular}[t]{l}$a$\end{tabular}}}}%
    \put(0,0){\includegraphics[width=\unitlength,page=4]{Slidingproof3.pdf}}%
    \put(0.46484341,0.07068118){\color[rgb]{0,0,0}\makebox(0,0)[lt]{\lineheight{1.25}\smash{\begin{tabular}[t]{l}$=\varepsilon$\end{tabular}}}}%
    \put(0,0){\includegraphics[width=\unitlength,page=5]{Slidingproof3.pdf}}%
    \put(0.66525674,0.01086025){\color[rgb]{0,0,0}\makebox(0,0)[lt]{\lineheight{1.25}\smash{\begin{tabular}[t]{l}$a$\end{tabular}}}}%
    \put(0.66525674,0.05405038){\color[rgb]{0,0,0}\makebox(0,0)[lt]{\lineheight{1.25}\smash{\begin{tabular}[t]{l}$a+1$\end{tabular}}}}%
    \put(0,0){\includegraphics[width=\unitlength,page=6]{Slidingproof3.pdf}}%
    \put(0.66568459,0.11539978){\color[rgb]{0,0,0}\makebox(0,0)[lt]{\lineheight{1.25}\smash{\begin{tabular}[t]{l}$a$\end{tabular}}}}%
    \put(0.780163,0.06692197){\color[rgb]{0,0,0}\makebox(0,0)[lt]{\lineheight{1.25}\smash{\begin{tabular}[t]{l}$-\varepsilon$\end{tabular}}}}%
    \put(0,0){\includegraphics[width=\unitlength,page=7]{Slidingproof3.pdf}}%
    \put(0.94464805,0.00859806){\color[rgb]{0,0,0}\makebox(0,0)[lt]{\lineheight{1.25}\smash{\begin{tabular}[t]{l}$a$\end{tabular}}}}%
    \put(0.94464805,0.05178819){\color[rgb]{0,0,0}\makebox(0,0)[lt]{\lineheight{1.25}\smash{\begin{tabular}[t]{l}$a-1$\end{tabular}}}}%
    \put(0,0){\includegraphics[width=\unitlength,page=8]{Slidingproof3.pdf}}%
    \put(0.94507586,0.11313757){\color[rgb]{0,0,0}\makebox(0,0)[lt]{\lineheight{1.25}\smash{\begin{tabular}[t]{l}$a$\end{tabular}}}}%
  \end{picture}%
\endgroup%

\end{center}
Combining the last two equations, we get the desired conclusion.
\end{proof}

 \begin{named}{Lemma \ref{lemma:singleArcAnnuli}}With the conventions of Figure \ref{fig:fusionRules2}, and for any $\varepsilon,\mu\in \lbrace \pm 1 \rbrace$,
we have the following identity
\begin{center}
\def \svgwidth{.55\columnwidth}

\end{center}
where  the arc on the left hand side has swift number $t$, and
$\delta_{\varepsilon,\mu}$ is the Kronecker symbol. \end{named}
\begin{proof}
When the swift number $t$ is equal to zero, the identity in the lemma reduces to the bigon rule. Next we treat the case $t=\pm 1.$ Using the half-twist rule, for the following swift number $1$ arc, we check that
\begin{center}
\def \svgwidth{.63\columnwidth}
\begingroup%
  \makeatletter%
  \providecommand\color[2][]{%
    \errmessage{(Inkscape) Color is used for the text in Inkscape, but the package 'color.sty' is not loaded}%
    \renewcommand\color[2][]{}%
  }%
  \providecommand\transparent[1]{%
    \errmessage{(Inkscape) Transparency is used (non-zero) for the text in Inkscape, but the package 'transparent.sty' is not loaded}%
    \renewcommand\transparent[1]{}%
  }%
  \providecommand\rotatebox[2]{#2}%
  \newcommand*\fsize{\dimexpr\f@size pt\relax}%
  \newcommand*\lineheight[1]{\fontsize{\fsize}{#1\fsize}\selectfont}%
  \ifx\svgwidth\undefined%
    \setlength{\unitlength}{298.03375725bp}%
    \ifx\svgscale\undefined%
      \relax%
    \else%
      \setlength{\unitlength}{\unitlength * \real{\svgscale}}%
    \fi%
  \else%
    \setlength{\unitlength}{\svgwidth}%
  \fi%
  \global\let\svgwidth\undefined%
  \global\let\svgscale\undefined%
  \makeatother%
  \begin{picture}(1,0.20471846)%
    \lineheight{1}%
    \setlength\tabcolsep{0pt}%
    \put(0,0){\includegraphics[width=\unitlength,page=1]{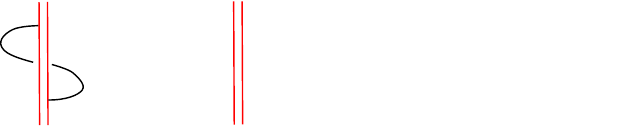}}%
    \put(0.080377,0.18805282){\color[rgb]{1,0,0}\makebox(0,0)[lt]{\lineheight{1.25}\smash{\begin{tabular}[t]{l}$a+\varepsilon$\end{tabular}}}}%
    \put(0.07859195,0.13837781){\color[rgb]{1,0,0}\makebox(0,0)[lt]{\lineheight{1.25}\smash{\begin{tabular}[t]{l}$a$\end{tabular}}}}%
    \put(0.08143492,0.01332946){\color[rgb]{1,0,0}\makebox(0,0)[lt]{\lineheight{1.25}\smash{\begin{tabular}[t]{l}$a+\mu$\end{tabular}}}}%
    \put(0.1309001,0.1008058){\color[rgb]{0,0,0}\makebox(0,0)[lt]{\lineheight{1.25}\smash{\begin{tabular}[t]{l}$=(-1)^{a+1}\varepsilon z^{\varepsilon}$\end{tabular}}}}%
    \put(0,0){\includegraphics[width=\unitlength,page=2]{SingleArcPattern3.pdf}}%
    \put(0.394749,0.17932471){\color[rgb]{1,0,0}\makebox(0,0)[lt]{\lineheight{1.25}\smash{\begin{tabular}[t]{l}$a+\varepsilon$\end{tabular}}}}%
    \put(0.39296395,0.12964969){\color[rgb]{1,0,0}\makebox(0,0)[lt]{\lineheight{1.25}\smash{\begin{tabular}[t]{l}$a$\end{tabular}}}}%
    \put(0.39580691,0.00460134){\color[rgb]{1,0,0}\makebox(0,0)[lt]{\lineheight{1.25}\smash{\begin{tabular}[t]{l}$a+\mu$\end{tabular}}}}%
    \put(0.48304227,0.08457762){\color[rgb]{0,0,0}\makebox(0,0)[lt]{\lineheight{1.25}\smash{\begin{tabular}[t]{l}$=(-1)^{a+1}z^{\varepsilon}\delta_{\varepsilon,\mu}$\end{tabular}}}}%
    \put(0,0){\includegraphics[width=\unitlength,page=3]{SingleArcPattern3.pdf}}%
    \put(0.7813137,0.12349281){\color[rgb]{1,0,0}\makebox(0,0)[lt]{\lineheight{1.25}\smash{\begin{tabular}[t]{l}$a+\varepsilon$\end{tabular}}}}%
  \end{picture}%
\endgroup%

\end{center}
The computation is similar for the other $3$ swift number $\pm 1$ patterns. Indeed, the other swift number $1$ is obtained from the previous one by rotation by $\pi$ along a vertical axis, hence we will get the same coefficient, while the swift number $-1$ patterns are obtained by mirror image from those ones, hence the coefficient we will get is the complex conjugate, which matches the lemma.
For a  pattern with swift number $\pm t$, we apply the strand rule $t-1$ times, to  get:
\begin{center}
\def \svgwidth{.4\columnwidth}
\begingroup%
  \makeatletter%
  \providecommand\color[2][]{%
    \errmessage{(Inkscape) Color is used for the text in Inkscape, but the package 'color.sty' is not loaded}%
    \renewcommand\color[2][]{}%
  }%
  \providecommand\transparent[1]{%
    \errmessage{(Inkscape) Transparency is used (non-zero) for the text in Inkscape, but the package 'transparent.sty' is not loaded}%
    \renewcommand\transparent[1]{}%
  }%
  \providecommand\rotatebox[2]{#2}%
  \newcommand*\fsize{\dimexpr\f@size pt\relax}%
  \newcommand*\lineheight[1]{\fontsize{\fsize}{#1\fsize}\selectfont}%
  \ifx\svgwidth\undefined%
    \setlength{\unitlength}{223.18262776bp}%
    \ifx\svgscale\undefined%
      \relax%
    \else%
      \setlength{\unitlength}{\unitlength * \real{\svgscale}}%
    \fi%
  \else%
    \setlength{\unitlength}{\svgwidth}%
  \fi%
  \global\let\svgwidth\undefined%
  \global\let\svgscale\undefined%
  \makeatother%
  \begin{picture}(1,0.43199073)%
    \lineheight{1}%
    \setlength\tabcolsep{0pt}%
    \put(0,0){\includegraphics[width=\unitlength,page=1]{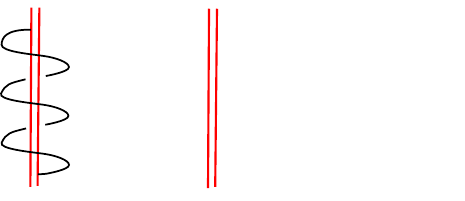}}%
    \put(0.09906658,0.3930824){\color[rgb]{1,0,0}\makebox(0,0)[lt]{\lineheight{1.25}\smash{\begin{tabular}[t]{l}$a+\varepsilon$\end{tabular}}}}%
    \put(0.09740723,0.32397335){\color[rgb]{1,0,0}\makebox(0,0)[lt]{\lineheight{1.25}\smash{\begin{tabular}[t]{l}$a$\end{tabular}}}}%
    \put(0.09076994,0.00752659){\color[rgb]{1,0,0}\makebox(0,0)[lt]{\lineheight{1.25}\smash{\begin{tabular}[t]{l}$a+\mu$\end{tabular}}}}%
    \put(0.47771694,0.39056429){\color[rgb]{1,0,0}\makebox(0,0)[lt]{\lineheight{1.25}\smash{\begin{tabular}[t]{l}$a+\varepsilon$\end{tabular}}}}%
    \put(0.15880247,0.19727341){\color[rgb]{0,0,0}\makebox(0,0)[lt]{\lineheight{1.25}\smash{\begin{tabular}[t]{l}$=\varepsilon^{t-1}$\end{tabular}}}}%
    \put(0,0){\includegraphics[width=\unitlength,page=2]{SingleArcPattern2.pdf}}%
    \put(0.52832674,0.30586512){\color[rgb]{1,0,0}\makebox(0,0)[lt]{\lineheight{1.25}\smash{\begin{tabular}[t]{l}$a+\varepsilon$\end{tabular}}}}%
    \put(0.4721066,0.23578545){\color[rgb]{1,0,0}\makebox(0,0)[lt]{\lineheight{1.25}\smash{\begin{tabular}[t]{l}$a+\varepsilon$\end{tabular}}}}%
    \put(0.52510726,0.16396056){\color[rgb]{1,0,0}\makebox(0,0)[lt]{\lineheight{1.25}\smash{\begin{tabular}[t]{l}$a+\varepsilon$\end{tabular}}}}%
    \put(0.46629894,0.09787036){\color[rgb]{1,0,0}\makebox(0,0)[lt]{\lineheight{1.25}\smash{\begin{tabular}[t]{l}$a+\varepsilon$\end{tabular}}}}%
    \put(0.46442813,0.01227521){\color[rgb]{1,0,0}\makebox(0,0)[lt]{\lineheight{1.25}\smash{\begin{tabular}[t]{l}$a+\mu$\end{tabular}}}}%
  \end{picture}%
\endgroup%

\end{center}
For each strand rule  applied, we only get one non zero contribution since from the previous computation, a single arc swift number $\pm 1$ pattern with opposite color shift on top and bottom vanishes. Now, simplifying each swift number $\pm 1$ pattern, we get the lemma.
\end{proof}

\appendix

\section{Metrification of a two-variable function}
\label{sec:metrification}

In this section, we introduce a general operation to produce a metric starting from a non-negative two variable function on a set.
We use the process to define the metric $d_{qt}$ in Section \ref{sec:TQFTmetric}.

Let $X$ be a set. Recall that $d:X\times X \longrightarrow [0,\infty)$ is called a metric on $X$ if it satisfies:
\begin{itemize}
\item[-](Symmetry) For any $x,y\in X,$ we have $d(x,y)=d(y,x).$
\item[-](Triangular inequality) For any $x,y,z \in X,$ we have $d(x,z)\leq d(x,y)+d(y,z).$ 
\item[-](Identity of indiscernables) For any $x,y \in X,$ $d(x,y)=0$ if and only if $x=y.$
\end{itemize}
We will call a function $d:X\times X \longrightarrow [0,\infty]$ a \textit{semi-metric} if it satisfies only the first two axioms and $d(x,x)=0$ for all $x\in X$. Notice we allow semi-metrics to sometimes take the value $\infty$.
\begin{definition}Let $X$ be a set and $f:X\times X \longrightarrow [0,\infty]$ be a function. For $x,y\in X$, let $P(x,y)$ be the set of all sequences ${\bf x}:=(x_0,x_1,\ldots , x_n)$ where $n\geq 0$ is an integer and $x_0=x,x_n=y.$ Then for any $x,y\in X,$ let
$$d_f(x,y)=\underset{{\bf x} \in P(x,y)}{\inf} \left( \min( f(x_0,x_1),f(x_1,x_0))+ \ldots + \min(f(x_{n-1},x_n),f(x_n,x_{n-1})) \right).$$
We call $d_f$ the metrification of $f.$
\end{definition}
The name \textit{metrification} is motivated by the following proposition:

\begin{proposition}\label{prop:metricproperties} Let $X$ be a set and $f:X\times X \rightarrow [0,\infty]$ be a function. Then:
\begin{itemize}
\item[(a)]$d_f$ is a semi-metric on $X$ and satisfies $d_f(x,y)\leq f(x,y)$ for any $x,y\in X.$
\item[(b)]For any semi-metric $d$ on $X$ such that $d\leq f$ on $X\times X,$ we have $d\leq d_f.$
\item[(c)]Suppose that for any $x,y\in X$ there is a finite sequence $\mathbf{x}:=(x_0,x_1,\ldots, x_n)$ in $P(x,y),$ such that for all $0\leq i \leq n-1,$ one has $f(x_i,x_{i+1})<\infty.$ Then $d_f$ takes value in $[0,\infty)$.
\item[(d)] Suppose that the hypothesis of (c) holds and furthermore that there is a constant $B>0$ so that for any $x\neq y$,  $f(x,y)\geq B$. Then $d_f$ is a metric on $X$.
\item[(e)]Assume that there is a group $G$ acting on $X$ so that $f$ is $G$-invariant. Then $d_f$ is $G$-invariant.
\item[(f)] If $f,g :X\times X \rightarrow [0,\infty]$ are functions such that $f\leq g$ on then $d_f\leq d_g$.
\end{itemize}
\end{proposition}
\begin{proof}
(a) For any $x,y \in X,$ we have $(x,y)\in P(x,y),$ so $$d_f(x,y)\leq \min(f(x,y),f(y,x))\leq f(x,y).$$ 
Moreover, for any $x\in X,$ the $1$-tuple $(x) \in P(x,x),$ therefore $d_f(x,x)=0.$
 
Next we show that $d_f$ is symmetric. For any $(x_0,x_1,\ldots ,x_n)\in P(x,y),$ we have that $(x_n,x_{n-1},\ldots,x_0) \in P(y,x).$ Therefore
\begin{multline*}d_f(y,x)\leq \min(f(x_n,x_{n-1}),f(x_{n-1},x_n))+ \ldots + \min(f(x_1,x_0),f(x_0,x_1))
\\ = \min( f(x_0,x_1),f(x_1,x_0))+ \ldots + \min(f(x_{n-1},x_n),f(x_n,x_{n-1})).
\end{multline*}
Taking the infimum over $P(x,y)$ we get $d_f(y,x)\leq d_f(x,y).$ By symmetry, $d_f(y,x)=d_f(x,y).$ 

Finally, $d_f$ satisfies the triangle inequality. Indeed, for any $x,y,z \in X,$ for any $(x=x_0,\ldots,x_n=y) \in P(x,y)$ and $(y_0=y,\ldots, y_m=z) \in P(y,z),$ we have that $(x=x_0,x_1,\ldots,x_n=y_0,y_1 \ldots, y_m=z) \in P(x,z).$ Thus
\begin{multline*}d_f(x,z)\leq \left( \min(f(x_0,x_1),f(x_1,x_0))+ \ldots + \min(f(x_{n-1},x_n),f(x_n,x_{n-1}) \right) 
\\+ \left( \min(f(y_0,y_1),f(y_1,y_0))+ \ldots + \min( f(y_{m-1},y_m), f(y_m,y_{m-1}))  \right).
\end{multline*}
Taking infimum over $(x_0,\ldots x_n) \in P(x,y)$ and $(y_0,\ldots,y_m)\in P(y,z),$ we get $d_f(x,z)\leq d_f(x,y)+d_f(y,z).$

(b) Assume that $d$ is a semi-metric and $d\leq f$ on $X\times X.$ For any $x,y,$ as $d$ is symmetric we have $d(x,y)=\min(d(x,y),d(y,x)) \leq \min(f(x,y),f(y,x)).$ Moreover, by the triangle inequality, for any $(x_0,\ldots ,x_n) \in P(x,y),$ we have
\begin{align}\notag
d(x,y)&\leq d(x,x_1)+\ldots +d(x_{n-1},y)&\\ &\leq \min(f(x,x_1),f(x_1,x)) + \ldots + \min(f(x_{n-1},y),f(y,x_{n-1})).\notag\end{align}
Taking the infimum over $P(x,y),$ we get $d(x,y)\leq d_f(x,y).$

(c) Let $x,y\in X$ and let $x_0,\ldots,x_n$ be the sequence provided by the hypothesis. By definition of $d_f$, we get that $d_f(x,y)\leq f(x_0,x_1)+\ldots +f(x_{n-1},x_n),$ which is finite, which proves the claim.

(d) First remark that $d_f$ takes value in $[0,\infty)$ by (c). Moreover, if $x\neq y,$ any element $(x_0=x,x_1,\ldots ,x_n=y)$ of $P(x,y)$ contains two consecutive points $x_i,x_{i+1}$ that are distinct. Then we have
\begin{align} \notag d_f(x,y) =&\min(f(x_0,x_1),f(x_1,x_0))+\ldots + \min(f(x_{n-1},x_n),f(x_n,x_{n-1}))\geq &\\ \notag
&\geq \min(f(x_i,x_{i+1}),f(x_{i+1},x_i))\geq B. \notag
\end{align}

(d) For any $g\in G,$ for any $x,y\in X$ and $(x_0=x,\ldots ,x_n=y) \in P(x,y),$ we have that $(g\cdot x_0,\ldots ,g\cdot x_n) \in P(g\cdot x,g\cdot y).$ Moreover, as $f$ is $G$-invariant, we have
\begin{multline*}\min(f(g\cdot x_0,g\cdot x_1),f(g\cdot x_1,g\cdot x_0))+\ldots +\min(f(g\cdot x_{n-1},g\cdot x_n),f(g\cdot x_n,g\cdot x_{n-1}))
\\=\min(f(x_0,x_1),f(x_1,x_0))+\ldots +\min(f(x_{n-1}, x_n),f(x_n,x_{n-1})).
\end{multline*}
Taking infimum over $P(x,y),$  we get $d_f(g\cdot x,g\cdot y) \leq d_f(x,y),$
and by symmetry, $d_f(g\cdot x,g\cdot y)=d_f(x,y).$
Finally, (e) is clear from the definition of $d_f,d_g.$
\end{proof}

\begin{figure}[!h]
\centering
\def \svgwidth{.70\columnwidth}
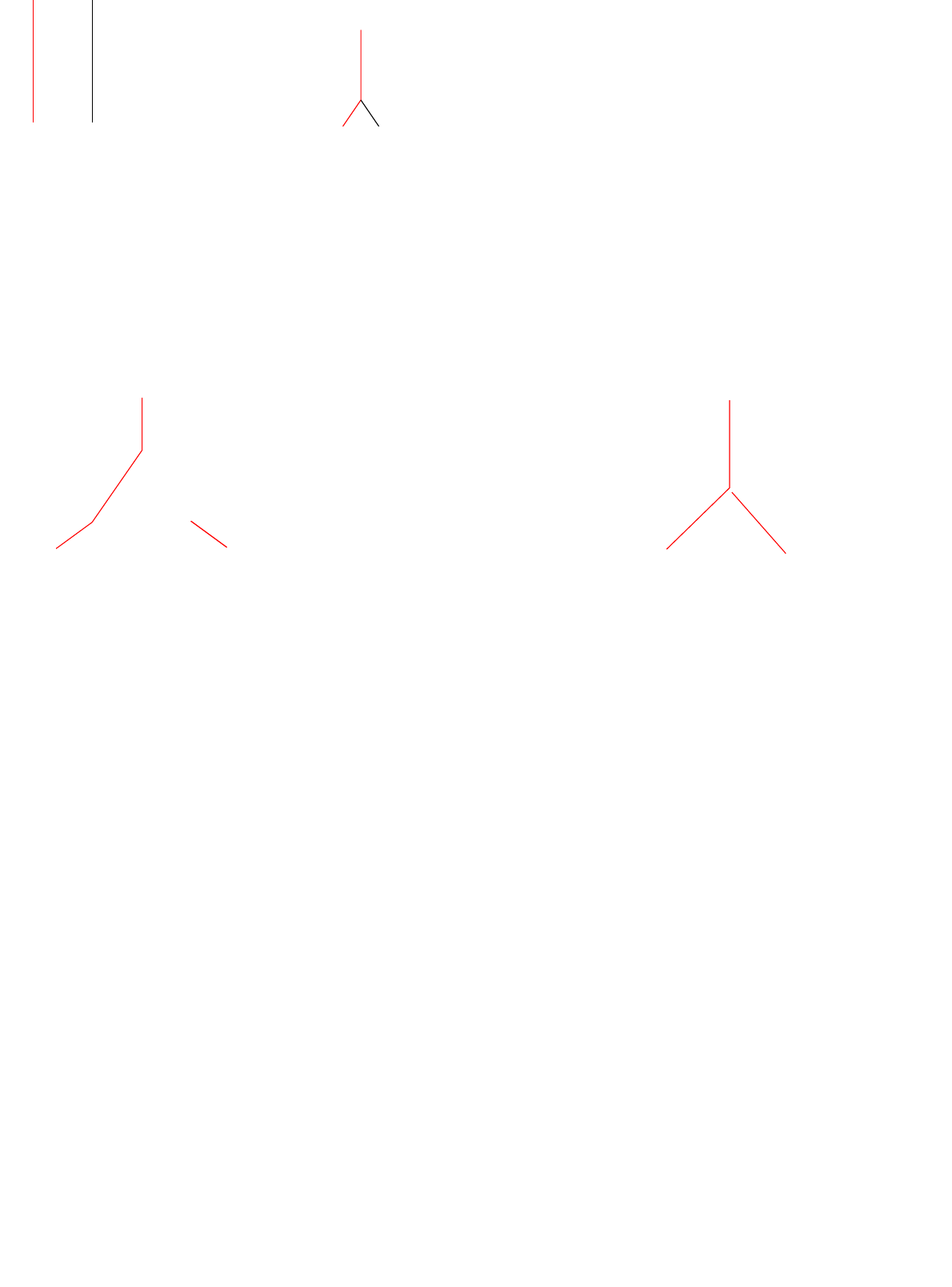
\caption{The fusion rules.}
\label{fig:fusionRules}
\end{figure}

\section{General fusion rules}
 \label{sec:appendix}
Fix a decomposition system $P\cup P'$ on $\Sigma$ with dual  graph $\Gamma$
and let
$\gamma$ be a multicurve  in Dehn-Thurston position
with respect to $P\cup P'$. 
By \cite{RD}, the 
coefficients $G_{\bf k}^{\gamma}$ of the operators $T^{\gamma}_r$
can be computed using fusion rules of Figure \ref{fig:fusionRules}, which are
an adaptation of the  Masbaum-Vogel \cite{MV94} $SU(2)$-fusion rules
in our setting. Note that, as in \cite{RD}, and in contrast with \cite{MV94, BHMV2} we  are using the \textit{normalized fusion rules}, since, as explained in Section \ref{sec:two},
the bases
${\mathcal B}_P$ of the TQFT spaces are normalized so that they are orthonormal (i.e. the vectors in them have Hermitian norm $1$). 
The difference between the two versions of the fusion rules is explained in detail in \cite[Section 4.1]{RD}; page 3072.

As we explained in  Section \ref{sec:two}
surface $\s$ is viewed as being constructed by glueing together two copies
of $\Gamma$, say $\Gamma, \Gamma'$. 
The red thick lines  in Figure \ref{fig:fusionRules} (shown in grey in black-and-white/grey scale)
indicate portions of the graph $\Gamma$ while the copy $\Gamma'$  is omitted.
We consider  $\Gamma \cup \gamma$
on $\s$ and the 
intersections of $\gamma$ with  $\Gamma'$ are recorded as over/under crossings of $\gamma$ and $\Gamma$.
The labels indicate admissible colorings of $\Gamma$
which correspond to vectors in the basis
of the TQFT spaces $V_r(\Sigma)$ obtained from $P$.

Note that all the relations occur between portions of colored copies of $\Gamma \cup \gamma$ that lie on a single pants or annulus piece  of the fixed decomposition of $\s$.
Reading Figure \ref{fig:fusionRules} from top to bottom, and rows from left to right,
we refer to these rules as follows: Strand rule, positive/negative half-twist rules, positive/negative/mixed triangle rules and bigon rule.
 
 The black thiner lines indicate  arcs of $\gamma$ and the color on these portions is $2$. 
 As we  remarked in   Section \ref{sec:two} the colors of all arcs have been shifted by $1$ from these of
  \cite{BHMV2}. Thus, in particular,
  the color $2$ corresponds to the fundamental $U_q(sl_2)$-representation of dimension $2$.
 Finally, for
 $x\in \ZZ$,  we have
 $<x>:=\sin(\frac{\pi x}{r})$.

\bibliographystyle{hamsplain}
\bibliography{biblio}
\end{document}